\documentclass[preprint,3p,number]{elsarticle}
\usepackage[utf8]{inputenc}
\usepackage{amsmath}
\usepackage{amssymb}
\usepackage{amsthm}
\usepackage{amsfonts}
\usepackage{bm}
\usepackage{titlesec}
\usepackage{mathtools}
\usepackage{esvect}
\usepackage{ stmaryrd }
\usepackage[ruled,vlined]{algorithm2e}
\usepackage{tikz}
\usepackage{graphicx}
\usepackage{caption}
\usepackage{subcaption}
\usepackage{ textcomp }
\usepackage{media9}
\usepackage{amsthm}
\usepackage{floatrow}
\usepackage{hyperref}
\usepackage{cleveref}
\usepackage{hhline}
\usepackage[titletoc]{appendix}
\usepackage{caption}
\usepackage[normalem]{ulem}
\usepackage{cancel}
\usepackage{lscape}
\usepackage{mdframed}
\usepackage{chngcntr}
\usepackage{titlesec}
\usetikzlibrary{positioning}

\newfloatcommand{capbtabbox}{table}[][\FBwidth]

\theoremstyle{definition}

\newtheorem{theorem}{Theorem}[section]

\DeclareMathOperator*{\argmax}{arg\,max}

\let\kron\otimes
\let\vec\bm

\newcommand{\oned}{\rm 1D}
\newcommand{\low}{\rm L}
\newcommand{\high}{\rm H}

\newcommand{\vol}{\rm vol}
\newcommand{\surf}{\rm surf}

\newcommand{\td}[2]{\frac{{\rm d}#1}{{\rm d}{ {#2}}}} 
\newcommand{\pd}[2]{\frac{\partial#1}{\partial#2}}

\newcommand{\nor}[1]{\left\| #1 \right\|} 
\newcommand{\LRp}[1]{\left( #1 \right)} 
\newcommand{\LRs}[1]{\left[ #1 \right]} 
\newcommand{\LRc}[1]{\left\{ #1 \right\}} 
\newcommand{\LRl}[1]{\left. #1 \right|} 

\newcommand{\fnt}[1]{\bm{\mathsf{ #1}}}

\newcommand{\rzero}[1]{{\color{black}{#1}}}
\newcommand{\rone}[1]{{\color{black}{#1}}}
\newcommand{\rtwo}[1]{{\color{black}{#1}}}

\title{High order entropy stable discontinuous Galerkin spectral element methods through subcell limiting} 

\author{Yimin Lin}
\author{Jesse Chan}
\address{Department of Computational and Applied Mathematics, Rice University, 6100 Main St, Houston, TX, 77005}

\begin{document}


\pagenumbering{arabic}

\begin{abstract}
    Subcell limiting strategies for discontinuous Galerkin spectral element
    methods do not provably satisfy a semi-discrete cell entropy inequality. In this
    work, we introduce an extension to the subcell \rtwo{and monolithic convex limiting strategies}~\rtwo{\cite{pazner2021sparse,rueda2022subcell,rueda2023monolithic}} that
    satisfies the semi-discrete cell entropy inequality by formulating the limiting
    factors as solutions to an optimization problem. The optimization problem is
    efficiently solved using a deterministic greedy algorithm. We also discuss
    the extension of the proposed subcell limiting strategy to preserve general
    convex constraints. Numerical experiments confirm that the proposed limiting
    strategy preserves high-order accuracy for smooth solutions and satisfies
    the cell entropy inequality.
\end{abstract}

\maketitle

\section{Introduction}\label{sec:intro}

In computational fluid dynamics simulations, higher resolutions are increasingly
necessary for a wide range of applications~\cite{slotnick2014cfd}. In certain
scenarios, high-order accurate numerical methods are preferred over low-order
methods due to their improved accuracy per degree of freedom, and comparable
efficiency~\cite{wang2013high}. High-order discontinuous Galerkin
(DG) methods are notably well-suited for handling convection-dominated problems
and yield simple and efficient implementations due to
the inherent locality of many operations~\cite{hesthaven2007nodal}. Among DG
methods, the DG spectral element method (DGSEM) is one of the most
computationally efficient high-order discretization techniques, due to
the tensor-product structure of the associated operators.

Unfortunately, high-order DGSEM often encounter stability issues when solving
nonlinear hyperbolic conservation laws. These issues arise due to the loss of
nonlinear stability and ill-defined physical quantities such as negative
density and pressure in compressible flows. Traditional stabilization
techniques, such as filtering and artificial
viscosity~\cite{krivodonova2007limiters,persson2006sub} are commonly employed
in combination with DGSEM. However, most of these methods require heuristic
tuning of parameters and lack provable robustness and high-order convergence for
smooth solutions.

For high-order DG schemes, Zhang, Shu, and their colleagues introduced simple,
effective, and high-order accuracy preserving scaling limiters for systems of
conservation laws~\cite{zhang2010maximum,zhang2010positivity,zhang2012minimum}.
The core idea is to utilize a strong stability preserving Runge-Kutta (SSPRK)
time integrator to compute the average of the DG solution on each element, which
satisfies desired properties under a timestep condition. The limited solution is
then constructed by scaling the high-order DG towards the DG average.

Another popular class of limiting strategies is flux-corrected transport (FCT)
algorithms~\cite{boris1973flux,zalesak1979fully}. Inspired by FCT, Kuzmin
designed algebraic flux correction (AFC) schemes, which provide a general
framework for designing multidimensional flux limiters~\cite{kuzmin2012flux}.
AFC schemes introduce a novel artificial diffusion operator and a conservative
flux decomposition to generalize the limiting technique of FCT algorithms. The
underlying low-order method is a generalization of the local Lax-Friedrichs
(LLF) method to nodal finite element
discretizations~\cite{kuzmin2012flux,guermond2016invariant}. A subcell limiting
strategy is then proposed to allow the limiting factors to vary within an
element in the context of high-order Bernstein finite
element~\cite{lohmann2017flux}. Pazner later extended this subcell limiting
strategy to the DGSEM in a dimension-by-dimension
fashion~\cite{pazner2021sparse}.

In the context of subcell limiting, there are two major approaches to ensure
some types of entropy inequality. One approach involves enforcing a discrete minimum principle
on specific entropy~\cite{pazner2021sparse,rueda2022subcell}. However, this
approach is limited to compressible flows and may reduce the accuracy of the
solution to at most second order for smooth solutions. Another recent approach
is to enforce Tadmor's entropy condition on subcell algebraic
fluxes~\cite{kuzmin2022limiter,rueda2023monolithic}.
This approach also does not preserve high-order accuracy for DGSEM\footnote{By private communication}.

This work aims to address the loss of high-order accuracy near smooth
regions when applying subcell limiter-based entropy stabilization. The main
motivation comes from high-order entropy stable Discontinuous Galerkin (ESDG)
discretizations \rone{and other schemes}, where a semi-discrete cell entropy balance is
satisfied~\cite{chan2018discretely,chan2019discretely,chan2022entropy,carpenter2014entropy,gassner2013skew,gassner2016split,chen2017entropy,abgrall2018general,abgrall2022reinterpretation}.
In contrast, the aforementioned entropy stabilization techniques enforce an
entropy stability inequality at the nodal level. The essential idea of the
proposed limiting strategy is to enforce the cell entropy balance through the
subcell limiting strategy, which can be formulated as a linear program over each
element. This linear program can be solved efficiently and optimally with
a simple greedy algorithm. Furthermore, the proposed entropy stabilization can
be easily extended to preserve more general convex constraints.

The outline of the paper is as follows: Section~\ref{sec:background} gives a
brief overview of the nonlinear conservation laws, the notations used in the
paper, and some background knowledge on DGSEM. Section~\ref{sec:subcell}
presents the semi-discrete entropy stable subcell limiting technique. In
Section~\ref{sec:exp}, we provide various numerical experiments in 1D and 2D to
verify the high-order convergence, entropy stability, and robustness of the
proposed limiting strategy. Finally, we summarize our work in
Section~\ref{sec:conclusion}.

\section{Background knowledge}~\label{sec:background}
In this work, we focus on solving the nonlinear hyperbolic conservation laws in
two space-dimensions, with the understanding that the theoretical findings
presented in this paper can be readily applied to three-dimensional settings.

\subsection{On notation}~\label{sec:notation}
We adopt the notation convention introduced in~\cite{chan2022entropy}. Lower and
upper case bold fonts (for example, $\vec{A}$ and $\vec{u}$) refer to vector and
matrix quantities, respectively. Spatially discrete
quantities are written in bold sans serif font (e.g., $\fnt{x}$). To
ensure clarity, continuous real functions evaluated over spatially discrete
quantities are interpreted as point-wise evaluations. For instance,

\begin{equation*}
    \fnt{x} = \begin{bmatrix}\vec{x}_1 \\ \vdots \\ \vec{x}_n\end{bmatrix},\qquad u: \mathbb{R} \rightarrow \mathbb{R}, \qquad u(\fnt{x}) = \begin{bmatrix} u(\vec{x}_1) \\ \vdots \\ u(\vec{x}_n)\end{bmatrix}
\end{equation*}

We note that we abuse notation and adopt the convention that $\fnt{A}\fnt{u}$
represents the Kronecker product $(\fnt{A}\kron \fnt{I}_{N^c})
\fnt{u}$~\cite{chen2017entropy}, where operator $\fnt{A}$ is applied to each
scalar component $\fnt{u}$. We will use either a number subscript
$\fnt{A}_1,\fnt{A}_2$, or a letter subscript $\fnt{A}_r$,$\fnt{A}_s$
interchangeably to indicate the coordinates of discrete operators. For the sake
of notational clarity, this work will present the theory on the reference
element $\widehat{D}$ and will ignore the geometric terms involved. For more
information on how the proposed framework can be applied to mapped elements and
curved meshes, readers can look into Appendix A of our recent
manuscript~\cite{lin2023positivity}.

For DG discretizations, let $u(\bm{x})$ be a scalar function on an element
$D$. We define $u$ as its ``interior'' and $u^+$ as its ``exterior'' values across
the face shared by neighbor $D^{+}$.

\subsection{Nonlinear conservation laws}~\label{sec:conslaw}
A $d$-dimensional nonlinear hyperbolic conservation law with $N_c$ components
is given by:
\begin{equation}\label{eq:conslaw}
    \pd{\vec{u}}{t} + \nabla \cdot \vec{f}\LRp{\vec{u}} = 0, \qquad \vec{u} \in \mathbb{R}^{N_c}, \qquad \vec{f}_i\LRp{\vec{u}}: \mathbb{R}^{N_c} \rightarrow \mathbb{R}^{N_c}, \quad i = 1,\dots,d.
\end{equation}
In particular, we are interested in systems with an associated convex
mathematical entropy $\eta\LRp{\vec{u}}$, \rone{whose physically relevant
solutions (defined as the limit solutions for an appropriately defined
vanishing viscosity) satisfy an entropy inequality}:
\begin{equation}\label{eq:esstate}
    \pd{\eta\LRp{\vec{u}}}{t} + \nabla\cdot \vec{F}\LRp{\vec{u}} \leq 0,
\end{equation}
where $\vec{F}\LRp{\vec{u}}$ is referred to as the entropy flux \rone{satisfying
the following identity:
\begin{equation}\label{eq:entropyfluxandvar}
  \LRp{\nabla_{\vec{u}} \vec{F}\LRp{\vec{u}}}^T = \vec{v}^T \nabla_{\vec{u}} \vec{f}\LRp{\vec{u}}, \qquad \vec{v} = \nabla_{\vec{u}} \eta\LRp{\vec{u}}.
\end{equation}
$\vec{v}$ is referred to as the entropy variables.} A cell entropy
balance is obtained by integrating the entropy
inequality~\eqref{eq:esstate} over a domain $D$ and applying integration by
parts and chain rule:
\begin{equation}\label{eq:integratedesstate}
    \int_{D} \pd{\eta\LRp{\vec{u}}}{t} + \int_{\partial D} \vec{v}^T \vec{f}\LRp{\vec{u}\LRp{\vec{v}}} - \vec{\psi}\LRp{\vec{v}} \leq 0,
\end{equation}
\rone{where $\vec{\psi}(\vec{v}) = \vec{v}^T \vec{f}(\vec{u}\LRp{\vec{v}}) - F(\vec{u}\LRp{\vec{v}})$
is referred to as the entropy potential.}

\subsection{Discretizations}~\label{sec:discretization}
The subcell limiting strategy is based on blending a high-order accurate
discretization and a low-order structure-preserving discretization constructed
using algebraic viscosity. In this section, we will provide a brief
introduction to the two types of discretizations that form the basis of the
proposed limiting strategy. We will restrict ourselves to 2D for simplicity of
presentations, but the idea is straightforward to extend to 3D. For both
discretizations, the domain $\Omega$ is decomposed into non-overlapping
quadrilateral elements $D^k$, each of which is the image of a reference element
$\hat{D}$ under an invertible mapping $\bm{\Phi}^k$. The reference approximation
basis of degree $N$ is defined as the Lagrange basis on Legendre-Gauss-Lobatto
(LGL) quadrature nodes
$\LRc{r_i}_{i=1}^{N+1}$:
\begin{align}\label{eq:lagrange}
    \phi_{i,j}\LRp{r,s} = \rone{L_i}\LRp{r}\rone{L_i}\LRp{s}, \qquad \rone{L_i}\LRp{r} = \prod\limits_{j \neq i} \frac{r-r_i}{r_j-r_i}.
\end{align}
In this work, we concentrate on Cartesian grids for the sake of simplicity. The
extension to curvilinear meshes is discussed in~\cite{pazner2021sparse}
and~\cite{lin2023positivity}.

\subsubsection{Discontinuous Galerkin spectral element discretization}\label{sec:DGSEM}
The DGSEM discretization refers to DG discretizations whose underlying
approximation basis is the Lagrange basis on LGL nodes and lumped mass matrix is
defined with LGL quadrature weights. Since the quadrature nodes collocate with
the interpolation nodes, the mass, differentiation, weighted differentiation,
boundary integration, and face extrapolation matrices are defined as
\begin{align}
    \fnt{M}_{\rm 1D} &= \begin{bmatrix} w_1 & & \\ & \ddots & \\ & & w_{N+1}\end{bmatrix}, \qquad \LRp{\fnt{D}_{\rm 1D}}_{ij} = \LRl{\td{\rone{L_j}}{r}}_{r=r_j},\qquad \fnt{Q}_{\rm 1D} = \fnt{M}_{\rm 1D}\fnt{D}_{\rm 1D}\label{eq:MDQ1D}\\
    \fnt{B}_{\rm 1D} &= \begin{bmatrix}-1 & 0 & \cdots & 0 & 0 \\ 0 & 0 & \cdots & 0 & 1\end{bmatrix},\qquad \fnt{E}_{\rm 1D} = \begin{bmatrix}1 & 0 & \cdots & 0 & 0 \\ 0 & 0 & \cdots & 0 & 1\end{bmatrix}, \label{eq:BE1D}
\end{align} 
where $\LRc{w_i}_{i=1}^{N+1}$ are the LGL quadrature weights.

Multidimensional operators are defined based on the tensor product structure of
the reference approximation basis as follows 
\begin{align}
    \fnt{M} &= \fnt{M}_{\rm 1D} \kron \fnt{M}_{\rm 1D},\qquad \fnt{E} = \fnt{I}_{N+1} \kron \fnt{E}_{\rm 1D}\label{eq:ME}\\
    \fnt{Q}_r &= \fnt{M}_{\rm 1D} \kron \fnt{Q}_{\rm 1D},\qquad  \fnt{Q}_s = \fnt{Q}_{\rm 1D} \kron \fnt{M}_{\rm 1D}, \label{eq:Qrs}\\
    \fnt{B}_r &= \fnt{M}_{\rm 1D} \kron \fnt{B}_{\rm 1D},\qquad  \fnt{B}_s = \fnt{B}_{\rm 1D} \kron \fnt{M}_{\rm 1D}. \label{eq:Brs}
\end{align}
\rone{Since the mass matrix $\fnt{M}$ is diagonal, we define $\fnt{m}_i = \fnt{M}_{ii}$
for simplicity of notation.}

We can then write the discontinuous Galerkin spectral element discretization on
the reference element as: 
\begin{align}
    &\fnt{M}\td{\fnt{u}}{t} + \sum\limits_{k=1}^2\fnt{Q}_k \fnt{f}_k + \fnt{E}^T\fnt{B}_k \LRp{\fnt{f}_k^*\LRp{\fnt{u}_f,\fnt{u}_f^+} - \vec{f}_k\LRp{\fnt{u}_f}} = 0,\label{eq:nodalDG}\\
    &\fnt{f}_k = \vec{f}_k\LRp{\fnt{u}}, \qquad \fnt{u}_f = \fnt{E} \fnt{u}, \nonumber
\end{align}
where $\fnt{u}_f^+$ denotes the interface value at the neighboring element
interface. Readers can refer to \rone{Chapter 6 of}~\cite{hesthaven2007nodal} for its extention to
multiple elements.

\subsubsection{Low order semi-discrete entropy stable and positivity-preserving discretization}\label{sec:lowpp}
We then introduce a low-order discretization that preserves semi-discrete
entropy stability and the positivity of physical quantities, utilizing an
appropriate time integrator. To avoid over-dissipation when the approximation
degree $N$ increases, we employ sparse low-order
operators~\cite{pazner2021sparse}. Sparse low-order operators on LGL nodes are
derived by integrating the piecewise linear basis over
subcells~\cite{pazner2021sparse}:
\begin{align}
    \fnt{Q}^{\low}_{\rm 1D} = \begin{bmatrix}-\frac{1}{2} & \frac{1}{2} & & \\-\frac{1}{2} & 0 & \frac{1}{2} & \\ & -\frac{1}{2} & 0 & \frac{1}{2} & \\ & & & \ddots\end{bmatrix}, \qquad \fnt{Q}_r^{\low} = \fnt{I}_{N+1} \kron \fnt{Q}_{\oned}^{\low}, \qquad \fnt{Q}_s^{\low} = \fnt{Q}_{\oned}^{\low} \kron \fnt{I}_{N+1}. \label{eq:Qrslow}
\end{align}
Then the low order discretization can be written
as~\cite{lin2023positivity}~\footnote{We
didn't split $\fnt{\Lambda}$ along dimension $k$ in~\cite{lin2023positivity}. }:
\begin{align}\label{eq:lowpp}
    \fnt{M}\td{\fnt{u}}{t} &+ \sum\limits_{k=1}^2\LRp{\LRp{\fnt{Q}_k^{\low} - {\fnt{Q}_k^{\low}}^T}\circ \fnt{F}_k} \bm{1} - \LRp{\fnt{\Lambda}_k \circ \fnt{D}} \bm{1} + \fnt{E}^T\fnt{B}_k \fnt{f}_k^*\LRp{\fnt{u}_f,\fnt{u}_f^+} = 0\\
    \LRp{\fnt{F}_k}_{ij} &= \frac{1}{2}\LRp{\vec{f}_k\LRp{\fnt{u}_i} + \vec{f}_k\LRp{\fnt{u}_j}},\qquad \fnt{D}_{ij} = \fnt{u}_j - \fnt{u}_i,\qquad
    \fnt{\Lambda}_{k,ij} = \frac{1}{2}\nor{\fnt{n}_{k,ij}} \lambda_{\max}\LRp{\fnt{u}_i, \fnt{u}_j, \frac{\fnt{n}_{k,ij}}{\nor{\fnt{n}_{k,ij}}}}, \nonumber\\
    \fnt{n}_{r,ij} &= \begin{bmatrix}\LRp{\fnt{Q}_r^{\low} - {\fnt{Q}_r^{\low}}^T}_{ij} \\ 0\end{bmatrix}, \qquad \fnt{n}_{s,ij} = \begin{bmatrix}0 \\ \LRp{\fnt{Q}_s^{\low} - {\fnt{Q}_s^{\low}}^T}_{ij}\end{bmatrix}\nonumber
\end{align}
The low-order discretization can be interpreted as a finite volume scheme with a
local Lax-Friedrichs type flux on subcells induced by LGL nodes. This
discretization satisfies a semi-discrete entropy inequality~\cite{lin2023positivity}~\footnote{Thoerem
6.1~\cite{lin2023positivity} proves the semi-discrete entropy stability of the
low order discretization rather than a fully discrete entropy stability.}.
If time integration is performed using the strong stability preserving
Runge-Kutta (SSP-RK) method, where the solution at the next time step is a
convex combination of forward Euler updates, then, under a suitable time step
condition, the combination of the low-order discretization and SSP time
integrator is proven to be both positivity preserving and entropy stable, as
demonstrated in~\cite{lin2023positivity}.

In this work, we assume the numerical fluxes of both
discretizations~\eqref{eq:nodalDG} and~\eqref{eq:lowpp} are both local
Lax-Friedrichs fluxes:
\begin{equation}
    \fnt{f}_k^*\LRp{\fnt{u}_f,\fnt{u}_f^+,\widehat{\fnt{n}}} = \frac{1}{2}\LRs{\vec{f}\LRp{\fnt{u}_f}+\vec{f}\LRp{\fnt{u}_f^+}} - \frac{\lambda_{\max}\LRp{\fnt{u}_f,\fnt{u}_f^+,\widehat{\fnt{n}}}}{2} \widehat{\fnt{n}}\LRs{\fnt{u}_f^+ - \fnt{u}_f},\label{eq:LFflux}
\end{equation}
where $\lambda_{\max}\LRp{\fnt{u}_f,\fnt{u}_f^+,\widehat{\fnt{n}}}$ is defined
as the maximum wavespeed associated with the 1D Riemann problem. 

\section{An entropy stable subcell limiting strategy}~\label{sec:subcell}
In this section, we present the main contribution of this paper. We begin with
Section~\ref{sec:subcell1D}, where we introduce the core idea in 1D.
Specifically, we present a linear program (LP) formulation for determining
optimal subcell limiting parameters. In Section~\ref{sec:subcell2D}, we extend
the subcell limiting strategy to higher dimensions. We discuss the adaptation of
the proposed limiting strategy as a shock capturing strategy in
Section~\ref{sec:subcelles}. Additionally, we discuss efficient and robust
implementation techniques for the proposed limiting strategy in
Section~\ref{sec:implementation}.

\subsection{An entropy stable subcell limiting strategy in 1D}~\label{sec:subcell1D}
In this section, we will illustrate the core idea of the proposed limiting
strategy on the 1D reference element. The 1D algebraic subcell flux form of the
DGSEM and the low order updates are defined as follows:
\begin{alignat}{2}\label{eq:1Dsubcell}
  &\fnt{m}_i \td{\fnt{u}^{\high}_i}{t} = \fnt{r}_i^{\high} = \bar{\fnt{f}}_i^{\high} - \bar{\fnt{f}}_{i-1}^{\high}, \qquad &&\bar{\fnt{f}}_i^{\high} = \sum\limits_{j=1}^{i} \fnt{r}_j^{\high},\\
  &\fnt{m}_i \td{\fnt{u}^{\low}_i}{t} = \fnt{r}_i^{\low} = \bar{\fnt{f}}_i^{\low} - \bar{\fnt{f}}_{i-1}^{\low}, \qquad     &&\bar{\fnt{f}}_i^{\low}  = \sum\limits_{j=1}^{i} \fnt{r}_j^{\low},\qquad i = 1,\dots,N+1\\
  &\bar{\fnt{f}}_0^{\high} = \bar{\fnt{f}}_0^{\low} = -\fnt{f}^*\LRp{\fnt{u}_1,\fnt{u}_1^+}, \qquad &&\bar{\fnt{f}}_{N+1}^{\high} = \bar{\fnt{f}}_{N+1}^{\low} = -\fnt{f}^*\LRp{\fnt{u}_{N+1},\fnt{u}_{N+1}^+},\label{eq:1Dsubcellalgfluxboundary}
\end{alignat}
where $\fnt{u}^{\high}_i$ and $\fnt{u}^{\low}_i$ denote the DGSEM and the low
order update at node $i$, respectively~\footnote{Note that the proposed limiting
strategy can also be applied to entropy stable discontinuous Galerkin
discretizations by utilizing an alternative definition of the high-order
residual $\fnt{r}_i^{\high}$, as illustrated by equation (35)
in~\cite{lin2023positivity}.}. It should be noted that the
equalities~\eqref{eq:1Dsubcellalgfluxboundary} hold because the DGSEM and low
order updates are both locally conservative~\cite{pazner2021sparse} and use the
same local Lax-Friedrichs fluxes at cell interfaces. The subcell limited
solution can then be written as
\begin{align}
  \fnt{m}_i \td{\fnt{u}_i}{t} &= \underbrace{\LRs{l_i \bar{\fnt{f}}_i^{\high} + \LRp{1-l_i} \bar{\fnt{f}}_i^{\low}}}_{\bar{\fnt{f}}_i} - \underbrace{\LRs{l_{i-1} \bar{\fnt{f}}_{i-1}^{\high} + \LRp{1-l_{i-1}} \bar{\fnt{f}}_{i-1}^{\low}}}_{\bar{\fnt{f}}_{i-1}} ,\qquad i = 1,\dots,N+1,\label{eq:1Dsubcelllimsol}
\end{align}
where $l_i \in \LRs{0,1}$ are referred to as the subcell limiting factors,\rtwo{
and limited algebraic subcell fluxes $\bar{\fnt{f}}_i$ are convex
combinations of low-order and high-order algebraic fluxes. \eqref{eq:1Dsubcelllimsol}
is referred to as monolithic scheme~\cite{hajduk2021monolithic,rueda2023monolithic},
where we determine appropriate subcell limiting factors to ensure the numerical
solution to remain in an admissble convex set.}

\subsubsection{Subcell limiting for cell entropy stability}~\label{sec:subcellcelles}
The ESDG discretization ensures a semi-discrete cell entropy balance, while
preserving high-order accuracy for smooth
solutions~\cite{chan2019discretely,carpenter2014entropy}. This motivates us to
consider the cell entropy inequality as a sufficient condition for achieving both
entropy stability and high-order accuracy. Furthermore, the subcell limiting
approach~\eqref{eq:1Dsubcelllimsol} enables the use of spatially varying limiting
factors within a DG element, which is essential for enforcing the entropy
stability condition.

We will now focus on the technical details of enforcing the cell entropy
inequality using the subcell limiting approach~\eqref{eq:1Dsubcelllimsol}.  The
subcell limited solution can be decomposed into two components: contributions
with numerical fluxes, referred to as surface contributions, and contributions
without numerical fluxes, referred to as volume contributions:
\begin{align}\label{eq:qvolsurf}
  \fnt{M} \td{\fnt{u}}{t} = \bar{\fnt{q}}^{\vol} + \bar{\fnt{q}}^{\surf} = \begin{bmatrix}\bar{\fnt{f}}_1 \\ \bar{\fnt{f}}_2 -  \bar{\fnt{f}}_1 \\ \vdots \\ \bar{\fnt{f}}_{N} -  \bar{\fnt{f}}_{N-1} \\ - \bar{\fnt{f}}_{N}\end{bmatrix} + \begin{bmatrix}-\bar{\fnt{f}}_0 \\ 0 \\ \vdots \\ 0 \\ \bar{\fnt{f}}_{N+1}\end{bmatrix}.
\end{align}
If the subcell limited solution satisfies a cell entropy inequality~\cite{chan2018discretely}, the
resulting entropy estimate can also be divided into volume and surface
contributions:
\begin{align}\label{eq:esstatement1D}
  \fnt{v}^T \fnt{M} \td{\fnt{u}}{t}
  &\leq \underbrace{\LRs{\fnt{\psi}\LRp{\fnt{u}_{N+1}} - \fnt{\psi}\LRp{\fnt{u}_1}}}_{\fnt{P}^{\vol}} - \underbrace{\LRs{\fnt{v}_{N+1}^T \fnt{f}^*\LRp{\fnt{u}_{N+1}, \fnt{u}_{N+1}^+}-\fnt{v}_1^T \fnt{f}^*\LRp{\fnt{u}_1, \fnt{u}_1^+}}}_{\fnt{P}^{\surf}}.
\end{align}
This observation enables us to enforce the cell entropy inequality separately on
the volume and surface contributions:
\begin{align}
  \fnt{v}^T\bar{\fnt{q}}^{\vol} \leq \fnt{P}^{\vol}, \qquad \fnt{v}^T\bar{\fnt{q}}^{\surf} \leq -\fnt{P}^{\surf}
\end{align}
The following identity holds for the surface contribution:
\begin{align}\label{eq:surfcontribution}
  \fnt{v}^T \rone{\bar{\fnt{q}}^{\surf}} = - \fnt{v}_1^T \bar{\fnt{f}}_0 + \fnt{v}_{N+1}^T \bar{\fnt{f}}_{N+1} = \fnt{v}_{N+1}^T \fnt{f}^*\LRp{\fnt{u}_{N+1}, \fnt{u}_{N+1}^+}-\fnt{v}_1^T \fnt{f}^*\LRp{\fnt{u}_1, \fnt{u}_1^+} = -\fnt{P}^{\surf},
\end{align}
since the limited algebraic surface flux does not depend on 
limiting factors by~\eqref{eq:1Dsubcellalgfluxboundary}. Therefore, the only
step left is to enforce entropy stability using the volume contribution. In
other words, we want to find limiting factors $l_1,\dots,l_N$ that satisfy:
\begin{align}\label{eq:es1D}
  \fnt{v}^T \bar{\fnt{q}}^{\vol} = \fnt{v}^T\begin{bmatrix}\bar{\fnt{f}}_1\LRp{l_1} \\ \bar{\fnt{f}}_2\LRp{l_2} -  \bar{\fnt{f}}_1\LRp{l_1} \\ \vdots \\ \bar{\fnt{f}}_{N}\LRp{l_N} -  \bar{\fnt{f}}_{N-1}\LRp{l_{N-1}} \\ - \bar{\fnt{f}}_{N}\LRp{l_N}\end{bmatrix} = \sum\limits_{i=1}^N \LRp{\fnt{v}_i - \fnt{v}_{i+1}}^T \bar{\fnt{f}}_i\LRp{l_i}\leq \fnt{P}^{\vol},
\end{align}
where we abused the notation to emphasize that $\bar{\fnt{f}}_i$ has a linear
dependence on the limiting factor $l_i$:
\begin{align}\label{eq:fbarli}
  \bar{\fnt{f}}_i = l_i \bar{\fnt{f}}_i^{\high} + \LRp{1-l_i} \bar{\fnt{f}}_i^{\low}
\end{align}

In addition to ensuring entropy stability, it is often desirable to preserve
general convex constraints on the solutions, such as the positivity of
thermodynamic quantities~\cite{lin2023positivity} or TVD-like
bounds~\cite{rueda2022subcell}. Let $l^{\rm C}$ denote the limiting
factors that preserve the convex constraints. \rone{Readers should refer to~\cite{rueda2023monolithic}
for formulas for subcell limiting factors under different convex constraints. We assume each subcell
limiting factor satisfies the bound $l_i^C \in \LRs{0,1}$.} The objective is to find a subcell
limited solution that preserves the convex constraints while maintaining
discrete entropy stability~\eqref{eq:es1D}. We aim to minimize the difference
between the subcell limited solution and the DGSEM discretization to preserve
high-order accuracy. Mathematically, the problem of finding suitable subcell
limiting factors can be formulated as a linear program:
\begin{subequations}
\begin{align}
    \max\limits_{l_{i}} \quad & \sum\limits_{i=1}^{N} l_{i} \label{eq:1DLPobjective}\\
    \textrm{s.t.}             \quad & \sum\limits_{i=1}^N \LRp{\fnt{v}_i - \fnt{v}_{i+1}}^T \bar{\fnt{f}}_i\LRp{l_i}\leq \fnt{P}^{\vol}\label{eq:constraintes1D}\\
                              \quad & 0 \leq l_i \leq l^{\rm C}_i\label{eq:constraintconvex1D}
\end{align}\label{eq:1DLP}
\end{subequations}

\rone{
The constraints~\eqref{eq:constraintes1D} and~\eqref{eq:constraintconvex1D}
correspond to semi-discrete entropy stability and convex constraints,
respectively, and are linear with respect to the subcell limiting factors.
In~\eqref{eq:1DLPobjective}, we define the linear objective function as the
sum of subcell limiting factors. As a result, the optimization
problem~\eqref{eq:1DLP} is a linear program, and its solution can be efficiently
obtained. Maximizing the objective function~\eqref{eq:1DLPobjective} is equivalent to
finding the set of subcell limiting factors as close to $l^C$ as possible. This
ensures that the amount of limiting applied is as small as possible while satisfying
entropy stability constraints.
}

The optimal solution of the linear program satisfies the following properties,
as stated in Theorem~\ref{thm:subcell}. It is important to note that the 
optimal solution in this context refers to a subcell limited solution that
utilizes the optimal solution of the linear program as its limiting factors.
\begin{theorem}\label{thm:subcell}
  The linear program~\ref{eq:1DLP} is solvable, and the optimal solution to
  the linear program is locally conservative,
  satisfies a semi-discrete entropy stability~\eqref{eq:esstatement1D}, and preserves
  the convex constraints enforced by $l^{\rm C}$.
\end{theorem}

\begin{proof}
The solvability follows from the entropy stability of the low order
discretization~\eqref{eq:lowpp}:
\begin{align}\label{eq:solvability}
    \sum\limits_{i=1}^N \LRp{\fnt{v}_i - \fnt{v}_{i+1}}^T \bar{\fnt{f}}_i\LRp{0} = 
    \fnt{v}^T \LRs{- \LRp{\LRp{\fnt{Q}^{\low} - {\fnt{Q}^{\low}}^T} \circ \fnt{F}} \bm{1} + \LRp{\fnt{\Lambda} \circ \fnt{D}}\bm{1}} \leq \fnt{\psi}\LRp{\fnt{u}_{N+1}} - \fnt{\psi}\LRp{\fnt{u}_{1}}
\end{align}

The optimal solution is locally conservative by~\eqref{eq:1Dsubcelllimsol} and~\eqref{eq:1Dsubcellalgfluxboundary}:
\begin{align}\label{eq:localcons}
  \sum\limits_{i=1}^{N+1} \fnt{m}_i \td{\fnt{u}_i}{t} = \bar{\fnt{f}}_{N+1}\LRp{l_{N+1}} - \bar{\fnt{f}}_{0}\LRp{l_0} = \fnt{f}^* \LRp{\fnt{u}_1, \fnt{u}_1^+} - \fnt{f}^* \LRp{\fnt{u}_{N+1}, \fnt{u}_{N+1}^+}.
\end{align}

In addition, the optimal solution
satisfies entropy stability~\eqref{eq:esstatement1D} due to the satisfaction of
the constraint~\eqref{eq:constraintes1D}. The
constraint~\eqref{eq:constraintconvex1D} preserves the given convex constraints
due to convexity.
\end{proof}

\subsubsection{Efficient solution of the linear program}~\label{sec:efficientLP}
The linear program we have formulated can be solved using simplex methods,
although its specific structure allows us to exploit certain advantages. To
simplify the notation, we can represent the linear program~\ref{eq:1DLP} as
follows:
\begin{subequations}
\begin{align}
    \max\limits_{\fnt{x}} \quad & \sum\limits_{i=1}^{M}\fnt{x}_i \\
    \textrm{s.t.}         \quad & \fnt{a}^T \fnt{x} \leq b \\
                                & \fnt{0} \leq \fnt{x} \leq \fnt{U}
\end{align}\label{eq:simpleLP}
\end{subequations}
This type of linear program is known as a continuous knapsack problem, and can
be efficiently solved with the greedy
algorithm~\ref{alg:greedy}~\cite{dantzig1957discrete}. Theorem~\ref{thm:greedy}
concludes this section by showing the optimality of
Algorithm~\ref{alg:greedy}.
\begin{algorithm}[ht!]
    \caption{Greedy algorithm for linear program~\ref{eq:simpleLP}}\label{alg:greedy}
    \KwData{$\fnt{a}, b, \fnt{U} \geq \fnt{0}$}
    \KwResult{$\fnt{x}$ optimal solution of LP}
    $\fnt{x} = \fnt{U}$;\\
    $\mathcal{I} = \LRc{1,\dots,M}$;\\
    \While{$\fnt{a}^T\fnt{x} > b$} {
        $i = \argmax\limits_{j} \ \fnt{a}_j$\\
        $\mathcal{I} = \mathcal{I}\backslash\LRc{i}$;\\
        \eIf{$\sum\limits_{j \in \rzero{\mathcal{I}}} \fnt{a}_j \fnt{x}_j > b$}{
            $\fnt{x}_i = 0.0$;
        }{
            $\fnt{x}_i = \frac{b - \sum\limits_{j \in\mathcal{I}} \fnt{a}_j \fnt{x}_j}{\fnt{a}_i}$;\\
            \bf{break};
        }
    }
\end{algorithm}

\begin{theorem}\label{thm:greedy}
    The greedy algorithm~\eqref{alg:greedy}
    gives the optimal solution to the linear program~\ref{eq:simpleLP}. 
\end{theorem}

\begin{proof}
The optimality follows from a contradiction argument. Note that a similar proof
could be found in~\cite{dantzig1957discrete}. Suppose there is an optimal 
solution $\fnt{y} \neq \fnt{x}$, where $\fnt{x}$ is the solution of the greedy algorithm. Since
$\fnt{y}$ is optimal, $\fnt{y}_i = \fnt{U}_i$ for $i: \fnt{a}_i \leq 0$.
Without loss of generality, we assume $a_i > 0$ for all $i$ and
$\fnt{a}_1 > \fnt{a}_2 > \cdots > \fnt{a}_M > 0$. Let $j$ be the smallest index
that $\fnt{x}_j \neq \fnt{y}_j$, and let $\mathcal{I} = \LRc{k, \dots, M}$
be the set after finishing Algorithm~\ref{alg:greedy}. 

\begin{enumerate}
  \item \fbox{If $j \in \mathcal{I}$,} then $\fnt{x}_j = \fnt{U}_j$. By feasibility and
optimality of $\fnt{y}$, $\fnt{y}_j \leq \fnt{U}_j$ and $\fnt{y}_j \geq
\fnt{U}_j$. Then $\fnt{x}_j = \fnt{y}_j$, contradiction.

\item \fbox{If $j \notin \mathcal{I}$ and $\fnt{y}_j < \fnt{x}_j$,}
If $j = k-1$, we know $\fnt{y}_i = \fnt{x}_i = 0$
for $i < j$, and $\fnt{y}_i \leq \fnt{U}_i = \fnt{x}_i$ for $i > j$. This
contradicts to the optimality of $\fnt{y}$. If $j < k-1$, 
$\fnt{y}_{k-1} < \fnt{x}_{k-1} = 0$, $\fnt{y}$ is infeasible.

\item \fbox{If $j \notin \mathcal{I}$ and $\fnt{y}_j > \fnt{x}_j$,}
then by feasibility of $\fnt{y}$, there is $i > j$ s.t.
$\fnt{y}_i < \fnt{x}_i$, otherwise the constraint
$\fnt{a}^T \fnt{y} \leq b$ is not satisfied. We can construct a new
solution $\tilde{\fnt{y}}$, where $\tilde{\fnt{y}}_i = \fnt{y}_i + \epsilon, \tilde{\fnt{y}}_j = \fnt{y}_j - \frac{a_i}{a_j}\epsilon$
for sufficiently small $\epsilon > 0$. $\tilde{\fnt{y}}$ contradicts the
optimality of $\fnt{y}$ since $i > j \implies a_i < a_j$, as a result $\sum \tilde{\fnt{y}} = \sum \fnt{y}$.
\end{enumerate}
\end{proof}

\subsection{Entropy stable subcell limiting strategy in higher dimensions}~\label{sec:subcell2D}
\vspace{-0.5cm}
\subsubsection{Multidimensional subcell limiting in matrix form}~\label{sec:subcellreview}
We now discuss in details the limiting framework in higher dimensions. 
In this section, we reformulate the subcell limiting strategy
discussed in~\cite{rueda2022subcell,pazner2021sparse} in a matrix form, which
facilitates the discussion of the proposed limiting strategy. In 2D, the
DGSEM~\eqref{eq:nodalDG} and the low order updates~\eqref{eq:lowpp} can be
rewritten in an algebraic subcell flux form
\begin{align}
    \fnt{M}\td{\fnt{u}^{\high}}{t} = \sum\limits_{k=1}^2 \fnt{\Delta}_k \bar{\fnt{f}}_k^{\high}, \qquad \fnt{M}\td{\fnt{u}^{\low}}{t} = \sum\limits_{k=1}^2 \fnt{\Delta}_k \bar{\fnt{f}}_k^{\low}. \label{eq:subcellHL}
\end{align}
The difference operators $\fnt{\Delta}_k$ are defined using one dimensional
difference operator $\fnt{\Delta}_{\oned}$ of size $\LRp{N+1}\times\LRp{N+2}$
\begin{align}
    \fnt{\Delta}_{\oned} = \begin{bmatrix}
                    -1 & 1  &   &  &    \\ 
                       & -1 & 1 &  &    \\
                       &    & \ddots & \ddots &    \\
                       &    &   &  &    \\
                       &    &   & -1 & 1\end{bmatrix}
                = \underbrace{\begin{bmatrix}
                    0 & 1  &   &  &    \\ 
                       & -1 & 1 &  &    \\
                       &    & \ddots & \ddots &    \\
                       &    &   &  &    \\
                       &    &   & -1 & 0\end{bmatrix}}_{\fnt{\Delta}^{\vol}_{\oned}}
                   +\underbrace{\begin{bmatrix}
                    -1 &    &   &  &    \\ 
                       &    &   &  &    \\
                       &    & \ddots & &    \\
                       &    &   &  &    \\
                       &    &   &    & 1 \end{bmatrix}}_{\fnt{\Delta}^{\surf}_{\oned}}.\label{eq:deltamatrix}
\end{align}
The difference operator in multidimension is defined using the Kronecker product,
and can be decomposed into volume and surface contributions,
\begin{align}
    \fnt{\Delta}_x &= \fnt{\Delta}_x^{\vol} + \fnt{\Delta}_x^{\surf} = \fnt{I}_{N+1} \kron \fnt{\Delta}^{\vol}_{\oned} + \fnt{I}_{N+1} \kron \fnt{\Delta}^{\surf}_{\oned}\label{eq:diffx}\\
    \fnt{\Delta}_y &= \fnt{\Delta}_y^{\vol} + \fnt{\Delta}_y^{\surf} = \fnt{\Delta}^{\vol}_{\oned} \kron \fnt{I}_{N+1} + \fnt{\Delta}^{\surf}_{\oned} \kron \fnt{I}_{N+1}\label{eq:diffy}
\end{align}

Because the high order and low order methods have the same numerical
fluxes~\eqref{eq:LFflux}, along each dimension, the surface contribution can
be written as
\begin{align}
    \fnt{\Delta}_k^{\surf}\bar{\fnt{f}}^{\high}_k = \fnt{\Delta}_k^{\surf}\bar{\fnt{f}}^{\low}_k
    = -\fnt{E}^T \fnt{B}_k \fnt{f}_k^{*},\label{eq:dsurff}
\end{align}
and the volume contribution can be written as
\begin{align}
    \fnt{\Delta}_k^{\vol}\bar{\fnt{f}}^{\high}_k
    &= - \sum\limits_{k=1}^2\fnt{Q}_k \fnt{f}_k + \fnt{E}^T\fnt{B}_k \vec{f}_k\LRp{\fnt{u}_f} \label{eq:dvolH}\\
    \fnt{\Delta}_k^{\vol}\bar{\fnt{f}}^{\low}_k
    &= - \sum\limits_{k=1}^2\LRp{\LRp{\fnt{Q}_k^{\low} - {\fnt{Q}_k^{\low}}^T}\circ \fnt{F}_k} \bm{1} + \LRp{\fnt{\Lambda}_k \circ \fnt{D}} \bm{1}\label{eq:dvolL}
\end{align}
It should be noted that each DG element has $\LRp{N+1}^2$
nodes, but there are $\left(N+2\right)\left(N+1\right)$ algebraic subcell fluxes
in each dimension. Therefore, equations~\eqref{eq:dvolH} and~\eqref{eq:dvolL}
form an underdetermined system of equations with respect to the algebraic
subcell fluxes. However, it has been shown that
equations~\eqref{eq:dsurff},~\eqref{eq:dvolH}, and~\eqref{eq:dvolL} are
well-defined by utilizing the dimension-by-dimension conservation property of
DGSEM and the low-order solution (e.g., Proposition 8
in~\cite{pazner2021sparse} and~\cite{mateo2023flux}).

Then, the subcell limited solution is constructed as
\begin{align}\label{eq:subcelllimsol}
    \fnt{M}\td{\fnt{u}}{t} = \sum\limits_{k=1}^2 \fnt{\Delta}_k \bar{\fnt{f}}_k,\qquad \bar{\fnt{f}}_{k,ij} = l_{k,ij} \bar{\fnt{f}}^{\high}_{k,ij} + \LRp{1-l_{k,ij}} \bar{\fnt{f}}^{\low}_{k,ij}
\end{align}
where the limited subcell fluxes $\bar{\fnt{f}}_k$ are convex combinations of
low and high order algebraic fluxes. Equivalently, in a matrix notation:
\begin{align}\label{eq:subcelllimsolmatrix}
    \fnt{M}\td{\fnt{u}}{t} = \sum\limits_{k=1}^2 \fnt{\Delta}_k\LRs{\fnt{L}_k \bar{\fnt{f}}^{\high}_k + \LRp{\fnt{I}-\fnt{L}_k} \bar{\fnt{f}}^{\low}_k}, \qquad \fnt{L}_k = \begin{bmatrix} l_{k,11} & & \\ & \ddots & \\ & & l_{k,N+2 N+1} \end{bmatrix}
\end{align}
Figure~\ref{fig:subcellflux} provides an illustration of the components of the
subcell limited solution.

\begin{figure}[!htb]
  \centering
  \includegraphics[width=.4\linewidth]{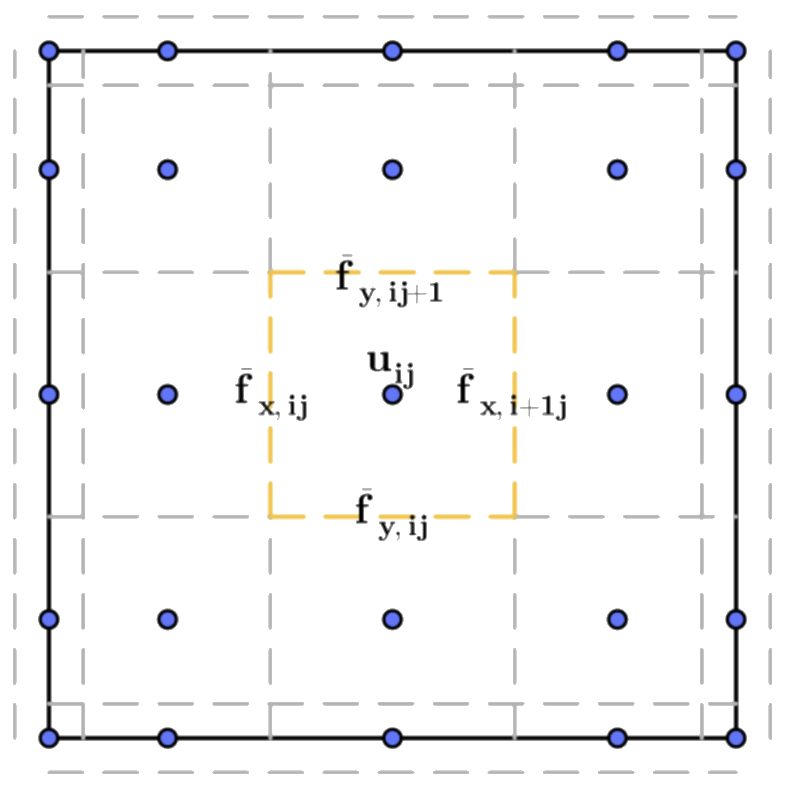}
  \caption{Illustration of algebraic subcell fluxes}\label{fig:subcellflux}
\end{figure}

\subsubsection{Enforcing cell entropy stability in higher dimensions}\label{sec:subcelles}

We proceed in a similar manner as in Section~\ref{sec:subcellcelles} and delve
into the discussion of the proposed subcell limiting strategy in the
multidimensional setting. In this context, the subcell limited solution can be
divided into surface and volume contributions by utilizing the difference
operators~\eqref{eq:diffx} and~\eqref{eq:diffy}. Additionally, the semi-discrete
cell entropy inequality~\cite{chan2018discretely} can also be decomposed into
volume and surface contributions:
\begin{align}
    \fnt{M} \td{\fnt{u}}{t} 
    &= \sum\limits_{k=1}^2 \underbrace{\fnt{\Delta}_k^{\vol} \bar{\fnt{f}}_k}_{\bar{\fnt{q}}^{\vol}_k} + \underbrace{\fnt{\Delta}_k^{\surf} \bar{\fnt{f}}_k}_{\bar{\fnt{q}}^{\surf}_k}, \label{eq:subcell2Dvolsurf}\\
    \fnt{v}^T \fnt{M} \td{\fnt{u}}{t} &\leq \sum\limits_{k=1}^2 \underbrace{\fnt{1}^T \fnt{B}_k \fnt{\psi}_k}_{\fnt{P}^{\vol}_k} - \underbrace{\LRp{\fnt{v}_f}^T \fnt{B}_k \fnt{f}_k^*}_{\fnt{P}^{\surf}_k}.\label{eq:esstatement2D}
\end{align}
In addition to enforce the cell entropy inequality separately on the volume and
surface contributions, we adopt a dimension-by-dimension approach to enforce the
inequality:
\begin{align}\label{eq:2Denforcement}
  \fnt{v}^T\bar{\fnt{q}}^{\vol}_k \leq \fnt{P}^{\vol}_k, \qquad \fnt{v}^T\bar{\fnt{q}}^{\surf}_k \leq -\fnt{P}^{\surf}_k, \qquad k = 1, 2.
\end{align}
Along each dimension $k$, the following identity holds for the surface
contributions by~\eqref{eq:dsurff}:
\begin{align}\label{eq:essurf2D}
  \fnt{v}^T\bar{\fnt{q}}^{\surf}_k = \fnt{v}^T \fnt{\Delta}_k^{\vol} \bar{\fnt{f}}_k = \LRp{\fnt{v}_f}^T \fnt{B}_k \fnt{f}_k^{*} = -\fnt{P}_k^{\surf}.
\end{align}
As a result, the semi-discrete entropy
balance~\eqref{eq:esstatement2D} holds for the subcell limited solution if the
volume contribution satisfies:
\begin{align}\label{eq:essubcellvol}
    \fnt{v}^T\bar{\fnt{q}}^{\vol}_k \leq \fnt{P}^{\vol}_k.
\end{align}
Through algebraic manipulations, we can write the inequality~\eqref{eq:essubcellvol}
explicitly in terms of the subcell limiting factors $\fnt{l}_k$ as unknowns
\begin{align}\label{eq:esineq}
    \fnt{l}_k^T \LRp{\underbrace{\LRp{{\fnt{\Delta}_k^{\vol}}^T \fnt{v}}^T \bar{\fnt{f}}^{\high}_k}_{\fnt{d}_k^{\high}} - \underbrace{\LRp{{\fnt{\Delta}_k^{\vol}}^T \fnt{v}}^T \bar{\fnt{f}}^{\low}_k}_{\fnt{d}_k^{\low}}} + \bm{1}^T\fnt{d}_k^{\low} \leq \fnt{1}^T \fnt{B}_k \fnt{\psi}_k,
\end{align}
Due to the sparsity pattern of $\fnt{\Delta}^{\vol}_k$, the number of unknowns
in the inequalities~\eqref{eq:esineq} can be reduced from $\LRp{N+2}\LRp{N+1}$ to $N\LRp{N+1}$.
In summary, we want to solve for volume subcell limiting factors in each dimension:
$\LRc{l_{x,ij}}_{i=1,\dots,N, j=1,\dots,N+1}$, $\LRc{l_{y,ij}}_{i=1,\dots,N+1,
j=1,\dots,N}$ that satisfies the entropy inequality~\eqref{eq:esineq}:
\begin{align}
    \sum\limits_{j=1}^{N+1} \sum\limits_{i=1}^N \LRp{\fnt{v}_{ij}-\fnt{v}_{i+1 j}}^T \LRp{\bar{\fnt{f}}_{x,i+1 j}^{\high}- \bar{\fnt{f}}_{x,i+1 j}^{\low}} \fnt{l}_{x,i+1j} \leq \bm{1}^T \fnt{B}_x \fnt{\psi}_x - \bm{1}^T\fnt{d}_x^{\low} \label{eq:esx}\\
    \sum\limits_{i=1}^{N+1} \sum\limits_{j=1}^N \LRp{\fnt{v}_{ij}-\fnt{v}_{i j+1}}^T \LRp{\bar{\fnt{f}}_{y,i j+1}^{\high}- \bar{\fnt{f}}_{x,i j+1}^{\low}} \fnt{l}_{y,i+1j} \leq \bm{1}^T \fnt{B}_y \fnt{\psi}_y - \bm{1}^T\fnt{d}_y^{\low} \label{eq:esy}
\end{align}

Let $\fnt{l}_{x}^{\rm C}$ denote the limiting factors that preserve the convex
constraints. Then, the problem of finding suitable subcell
limiting factors can be formulated as linear programs for each dimension $k$.
We present the linear program in x-direction for brevity:
\begin{subequations}
\begin{align}
    \max\limits_{\fnt{l}_{x,ij}} \quad & \sum\limits_{j=1}^{N+1}\sum\limits_{i=2}^{N+1} \fnt{l}_{x,ij} \label{eq:LP2Dobjective}\\
    \textrm{s.t.}                \quad & \sum\limits_{j=1}^{N+1} \sum\limits_{i=1}^N \LRp{\fnt{v}_{ij}-\fnt{v}_{i+1 j}}^T \LRp{\bar{\fnt{f}}_{x,i+1 j}^{\high}- \bar{\fnt{f}}_{x,i+1 j}^{\low}} \fnt{l}_{k,i+1j} \leq \bm{1}^T \fnt{B}_x \fnt{\psi}_x - \bm{1}^T\fnt{d}_x^{\low} \label{eq:LP2Dconstraintes}\\
                                       & \fnt{0} \leq \fnt{l}_{x,ij} \leq \fnt{l}_{x,ij}^{\rm C} \label{eq:LP2Dconstraintconvex}
\end{align}\label{eq:LP2D}
\end{subequations}
The linear program is still of form~\eqref{eq:simpleLP}, and the greedy
algorithm~\ref{alg:greedy} can still be applied. The solution of this linear
program~\eqref{eq:LP2D} satisfies the following properties:

\begin{theorem}\label{thm:LP2D}
    The linear program~\ref{eq:LP2D} is solvable. The optimal solution to the
    linear program is locally conservative,
    satisfies a semi-discrete entropy stability~\eqref{eq:esstatement2D}, and preserves
    the given convex constraints enforced by $l^{\rm C}$.
\end{theorem}

\begin{proof}
  The solvability follows from the entropy stability of the low order
  discretization:
  \begin{align}\label{eq:lowes}
    \bm{1}^T\fnt{d}_x^{\low} = \LRp{{\fnt{\Delta}_k^{\vol}}^T \fnt{v}}^T \bar{\fnt{f}}^{\low}_k \leq \bm{1}^T \fnt{B}_x \fnt{\psi}_x.
  \end{align}

  The optimal solution is locally conservative by simple algebra:
  \begin{align}\label{eq:subcellcons}
    \fnt{1}^T \fnt{M}\td{\fnt{u}}{t} = \sum\limits_{k=1}^2 \fnt{1}^T \fnt{\Delta}_k \bar{\fnt{f}}_k = \sum\limits_{k=1}^2 \fnt{1}^T \fnt{\Delta}_k^{\surf} \bar{\fnt{f}}_k = -\sum\limits_{k=1}^2\fnt{1}^T \fnt{B}_k \fnt{f}_k^*,
  \end{align}
  where we used the identity $\bm{1}^T \fnt{\Delta} = \bm{1}^T
  \fnt{\Delta}^{\surf}$, along with the assumption~\eqref{eq:LFflux} that high
  and low order methods share the same surface contributions. The semi-discrete
  entropy stability and preservation of convex constraints follows from the
  constraints~\eqref{eq:LP2Dconstraintes} and~\eqref{eq:LP2Dconstraintconvex}.
\end{proof}

\subsection{Incorporating shock capturing}\label{sec:shockcap}

At nonsmooth regions, the volume entropy estimate of the subcell limited
solution is expected to dissipate. To account for this, we propose an
alternative upper bound as a replacement for~\eqref{eq:LP2Dconstraintes} on the
volume entropy estimate of the subcell limited solution. This new bound
combines the original entropy estimate with the dissipative low order entropy
estimate:
\begin{align}\label{eq:esindicatorsubcell}
    \fnt{v}^T \fnt{\Delta}_k^{\vol} \bar{\fnt{f}}_k 
    &\leq \LRp{1-\epsilon} \LRs{\fnt{1}^T \fnt{B}_k \fnt{\psi}_k - \fnt{v}^T \fnt{\Delta}_k^{\vol} \bar{\fnt{f}}_k^{\low}}+ \epsilon \LRs{\LRp{1-\beta} \fnt{1}^T \fnt{B}_k \fnt{\psi}_k + \beta \fnt{v}^T \fnt{\Delta}_k^{\vol} \bar{\fnt{f}}_k^{\low} - \fnt{v}^T \fnt{\Delta}_k^{\vol} \bar{\fnt{f}}_k^{\low}} \nonumber\\
    &= \LRp{1-\beta\epsilon} \LRs{\fnt{1}^T \fnt{B}_k \fnt{\psi}_k - \fnt{v}^T \fnt{\Delta}_k^{\vol} \bar{\fnt{f}}_k^{\low}},
\end{align}
where $\epsilon \in [0,1]$ is an elementwise modal smoothness factor. \rone{In particular,
the smoothness factor is defined as~\cite{pazner2021sparse,hennemann2021provably}:
\begin{align}\label{eq:elemmodalsmooth}
  \epsilon = \begin{cases}
    0, &\text{if } s < s_0 - \kappa\\
    0.5\LRp{1- \sin\LRp{\pi \frac{s-s_0}{2\kappa}}} &\text{if } s \in \LRs{s_0 - \kappa, s_0 + \kappa}\\
    1, &\text{if } s < s_0 + \kappa\end{cases},\quad
    s = \log_{10} \LRp{\max\LRp{\frac{\mu_{N_p}^2}{\sum\limits_{i=0}^{N_p} \mu_i^2}, \frac{\mu_{N_p-1}^2}{\sum\limits_{i=0}^{N_p-1} \mu_i^2}}},
\end{align}
where $s$ is the modal smoothness indicator, and $N_p = \LRp{N+1}\LRp{N+1}$ is the
number of degrees of freedom in an element. We set user-defined parameters
$\kappa = 1, s_0 = \log_{10} \LRp{N^{-4}}$ as in~\cite{pazner2021sparse}. $\LRc{\mu_i}_{i=0}^{N}$ are the modal
coefficients of the polynomial solution in the orthonormal
Legendre basis $\LRc{\widetilde{L}_i\LRp{x}}_{i=0}^{N_p}$, such that
$u\LRp{x} = \sum\limits_{i=0}^N \mu_i \widetilde{L}_i \LRp{x}$.
} $\beta\in [0,1]$ is a \rone{user-defined}
parameter that determines the maximum amount of the low-order entropy estimate to
blend. When $\beta = 0$, the original bound~\eqref{eq:LP2Dconstraintes} is recovered. In smooth regions,
entropy is conserved in the interior of the element. In nonsmooth
regions, the amount of entropy dissipation added is guided by the entropy
dissipation from the low-order discretization.

\rone{
For completeness, we present here the linear program formulation to enforce cell
entropy stability with shock capturing:
\begin{subequations}
\begin{align}
    \max\limits_{\fnt{l}_{x,ij}} \quad & \sum\limits_{j=1}^{N+1}\sum\limits_{i=2}^{N+1} \fnt{l}_{x,ij} \label{eq:LP2Dwshockobjective}\\
    \textrm{s.t.}                \quad & \sum\limits_{j=1}^{N+1} \sum\limits_{i=1}^N \LRp{\fnt{v}_{ij}-\fnt{v}_{i+1 j}}^T \LRp{\bar{\fnt{f}}_{x,i+1 j}^{\high}- \bar{\fnt{f}}_{x,i+1 j}^{\low}} \fnt{l}_{k,i+1j} \leq \LRp{1 - \beta \epsilon}\LRs{\bm{1}^T \fnt{B}_x \fnt{\psi}_x - \bm{1}^T\fnt{d}_x^{\low}} \label{eq:LP2Dwshockconstraintes}\\
                                       & \fnt{0} \leq \fnt{l}_{x,ij} \leq \fnt{l}_{x,ij}^{\rm C} \label{eq:LP2Dwshockconstraintconvex}
\end{align}\label{eq:LP2Dwshockcap}
\end{subequations}

We note that the only difference between the LP formulations~\eqref{eq:LP2D}
and~\eqref{eq:LP2Dwshockcap} is the upper bound of the cell entropy inequality
constraint~\eqref{eq:LP2Dconstraintes} and~\eqref{eq:LP2Dwshockconstraintes}.
} In future sections, we refer to the condition~\eqref{eq:LP2Dwshockconstraintes} with
$\beta = 0$ as the cell entropy inequality condition. A shock capturing cell entropy
stability refers to the same condition with a nonzero value of $\beta$.

\subsection{On implementation}\label{sec:implementation}

The proposed greedy algorithm can be efficiently implemented by first sorting
the coefficient vector $\fnt{a}$, and then proceeding with the greedy iterations
in the sorted order. It can also be proven that for an index $i$ where
$\fnt{a}_i$ is non-positive, the optimal solution is $\fnt{x}_i = \fnt{U}_i$ due
to the non-negativity constraint on $\fnt{x}_i$. Furthermore, the number of
arithmetic operations in Algorithm~\ref{alg:greedy} can be reduced by \rone{iteratively accumulating the dot
product $\fnt{a}^T \fnt{x}$ in each greedy step instead of naively evaluating
the dot product at each step}. Overall,
the number of operations for Algorithm~\ref{alg:efficientgreedy} is
$O\LRp{m\log\LRp{m}}$, where $m = N\LRp{N+1}$. \rone{We note that
calculating the coefficient vector $\fnt{a}$ and the right hand side $b$
involves evaluating the entropy variables and entropy potentials, which may
be nonlinear functions with respect to the conservative variables $\fnt{u}$~\footnote{
For example, for the compressible Euler equations, calculating the entropy
variables involve evaluting a nonlinear transformation~\cite{lin2023positivity}.
The entropy potential is a scalar multiple of the momentum~\eqref{eq:eulerpsi}.
}.}

However, in practice, numerical issues can arise due to floating-point errors.
One of the issues arises due to the division by $\fnt{a}_i$. To avoid this, a
tolerance $\epsilon_0 > 0$, close to machine epsilon, is set, and indices $i$
where $\fnt{a}_i < \epsilon_0$ are skipped during greedy iterations.
\rone{We set $\epsilon_0 = 10^{-14}$ in subsequent numerical experiments.}

\rone{
In summary, Algorithm~\ref{alg:efficientgreedy} presents the efficient and robust
pseudocode implementation of the proposed limiting strategy. We summarize the
proposed subcell limiting strategy that preserves both cell entropy stability and
convex constraints in Algorithm~\ref{alg:eslimiting}:
}
\begin{algorithm}[ht!]
    \caption{Efficient implementation of Algorithm~\ref{alg:greedy}}\label{alg:efficientgreedy}
    \KwData{$\fnt{a}, b, \fnt{U} \geq \fnt{0}, \epsilon_0 > 0$}
    \KwResult{$\fnt{x}$ optimal solution of LP}
    $\fnt{x} = \fnt{U}$;\\
    Sort $\fnt{a}$ in decreasing order, \rzero{and permute $\fnt{x}$ with respect to the order of $\fnt{a}$};\\
    $s = \sum\limits_{j = 1}^{N\LRp{N+1}} \fnt{a}_j \fnt{x}_j$;\\
    \rzero{
    \If{$s \leq b$}{
      \bf{return};
    }
    }
    \For{$i = 1,\dots,N\LRp{N+1}$} {
        \If{$\fnt{a}_i < \epsilon_0$}{
            \bf{break};
        }
        \rzero{
        $s = s - \fnt{a}_i \fnt{x}_i$;\\
        \If{$s \leq b$}{
          \rzero{$\fnt{x}_{i} = \frac{b - s}{\fnt{a}_{i}}$};\\
          \bf{break};
        }
        \Else{
          $\fnt{x}_i = 0.0$;\\
        }
        }
    }
\end{algorithm}
\begin{algorithm}[ht!]
    \caption{The entropy stable and convex constraints preserving limiting strategy}\label{alg:eslimiting}
    \KwData{$\fnt{u}^{n}$ solution at the current time step}
    \KwResult{$\fnt{u}^{n+1}$ limited solution at the next time step}
    Compute $\fnt{u}^{\high}, \fnt{u}^{\low}$ by~\eqref{eq:nodalDG},~\eqref{eq:lowpp};\\
    \For{{\rm each dimension} $k$} {
        Compute algebraic subcell fluxes $\bar{\fnt{f}}_k^{\high}, \bar{\fnt{f}}_k^{\high}$ by~\eqref{eq:dsurff},~\eqref{eq:dvolH},~\eqref{eq:dvolL};\\
        Find limiting parameters $\fnt{l}_k^{\rm C}$ that satisfies the given convex constraints.\\
        Find new limiting parameters $\fnt{l}_k$ that satisfies the entropy stability using Algorithm~\ref{alg:efficientgreedy}.\\
        Construct the limited solution $\fnt{u}^{n+1}$ by~\eqref{eq:subcelllimsol}
    }
\end{algorithm}

\section{Numerical Experiments}~\label{sec:exp}
In this section, we present various numerical experiments to verify the entropy
stability, high-order convergence, and robustness of the proposed limiting
strategy~\footnote{The codes used for the experiments are available at
https://github.com/yiminllin/P2DE.jl}. Simulations advance in time with the optimal
third-order, \rone{three-stage} Strong Stability Preserving (SSP) Runge-Kutta,
\rone{usually referred to as the SSPRK(3,3)}, scheme~\cite{gottlieb2001strong}.
We choose the timestep size according to the timestep condition derived
in~\cite{lin2023positivity}~\footnote{The SSPRK time integrator
and the CFL condition~\eqref{eq:CFL} are necessary only for the preservation of
positivity and not for satisfying the cell entropy inequality. However, for the
sake of consistency, we choose to use this time integrator and CFL condition
across all numerical experiments.}:
\begin{align}\label{eq:CFL}
    \Delta t = \frac{1}{2} \min\limits_{i} \frac{\fnt{m}_i}{2 \lambda_i}.
\end{align}
\rone{We note that the time step condition~\eqref{eq:CFL}
is calculated with the solution at the first SSP stage and is used in all
three SSP stages.}

In this section, two nonlinear conservation laws are studied: the compressible
Euler equation and the KPP problem. To avoid repetition, readers can refer to a
previous manuscript~\cite{lin2023positivity} for the formula of the compressible
Euler equation, and to Section 5.1 of~\cite{rueda2022subcell} for the formula of
the KPP problem.

For the compressible Euler equations presented in~\cite{lin2023positivity}, the
entropy potential is defined as
\begin{align}\label{eq:eulerpsi}
    \vec{\psi}\LRp{\vec{u}} = \LRp{\gamma-1}\begin{bmatrix}\rho u \\ \rho v\end{bmatrix}. 
\end{align}
We estimate the maximum wavespeed associated with the 1D Riemann problems using the
Davis estimate~\cite{davis1988simplified}~\footnote{\rone{We note the Davis
estimate is not always an upper bound of the maximum
wavespeed as shown in~\cite{guermond2016fast}. As a result, the low order
discretization~\eqref{eq:lowpp} using the Davis estimate does not provably satisfy
the positivity-preservation and semi-discrete entropy stability. While we have
not observed issues using Davis estimate in our numerical experiments, we plan 
to explore more robust estimates of the maximum wavespeed~\cite{guermond2016fast}
in future work.}}:
\begin{align}
    \lambda_{\max} \LRp{\vec{u}_L,\vec{u}_R,\vec{n}} = \max\LRp{\left|\vec{u}_L \cdot \vec{n}\right| + \sqrt{\gamma \frac{p_L}{\rho_L}}, \left|\vec{u}_R \cdot \vec{n}\right| + \sqrt{\gamma \frac{p_R}{\rho_R}}}. \label{eq:maxwvspd}
\end{align}

For comparison, we will investigate some other entropy stabilization techniques
for subcell limiting. One popular approach is to enforce a local minimum
principle on a modified specific 
entropy~\cite{pazner2021sparse,rueda2022subcell}:
\begin{align}
    \phi_i \geq \min\limits_{j \in \mathcal{N}\LRp{i}} \phi_j^n, \qquad \phi\LRp{\vec{u}} = \rho^{1-\gamma} e, \label{eq:minentropy}
\end{align}
where $\mathcal{N}\LRp{i}$ denotes the set of indices $j$ such that the sparse
low order operator $\fnt{Q}_{ij}^{\low} - {\fnt{Q}_{ij}^{\low}}^T$ is nonzero.
We also consider a relaxed local minimum entropy
principle~\cite{pazner2021sparse}:
\begin{align}
    \phi_i \geq \epsilon \min\limits_{j \in \mathcal{N}\LRp{i}} \phi_j^n + \LRp{1-\epsilon} \phi_{\rm global}, \label{eq:minentropyrelax}
\end{align}
where $\epsilon \in \LRs{0,1}$ is the elementwise modal indicator. Another
entropy stabilization proposed in~\cite{kuzmin2022limiter} enforced a Tadmor's
entropy condition for subcell fluxes, and we refer this approach as the subcell 
entropy fix. 

\rone{In subsequent numerical experiments, we enforce inflow boundary
conditions by regarding them as Dirichlet boundary conditions, whose values are
given by the initial condition at inflow. Readers can refer to
Section 7.1.1 of~\cite{hesthaven2007nodal} for the implementation of Dirichlet
boundary condition for discontinuous Galerkin discretizations. No boundary
conditions are applied at the outflow.

For all numerical experiments, we assume $\beta = 0$
in~\eqref{eq:esindicatorsubcell} unless otherwise explicitly stated.}

\subsection{On entropy stability}~\label{sec:esexperiment}
In this section, we will investigate the entropy stability of the proposed
limiting strategy using two benchmark test cases to check any entropy violation.

\subsubsection{Modified Sod shocktube}~\label{sec:sod}
We will first consider the modified Sod shock tube
problem, \rone{which can be found in Section 6.4 of~\cite{toro2013riemann}}.
This problem consists of a left sonic
rarefaction wave and is useful for testing whether numerical solutions violate
the entropy condition. An entropy-satisfying solution should produce a smooth
rarefaction wave. The initial condition for this test problem is given on the
domain $\LRs{0,1}$:
\begin{align}
    \vec{u}\LRp{x} = \begin{cases}
        \vec{u}_L, \quad &x < 0.3\\
        \vec{u}_R, \quad &\text{otherwise}
    \end{cases}, \quad \vec{u}_L &= \begin{bmatrix}
    \rho_L \\
    u_L \\
    p_L \\
    \end{bmatrix} = \begin{bmatrix}
    1.0 \\
    0.75 \\
    1.0 \\
    \end{bmatrix},\qquad
    \vec{u}_R = \begin{bmatrix}
    \rho_R \\
    u_R \\
    p_R \\
    \end{bmatrix} = \begin{bmatrix}
    0.125 \\
    0.0 \\
    0.1 \\
    \end{bmatrix}, \label{eq:modifiedsod}
\end{align}
where we impose inflow boundary conditions at $x = 0.0$ and outflow boundary
conditions at $x = 1.0$.

We discretize the domain using $K = 50$ and $K = 100$ uniform intervals, with
polynomial degree $N = 3$, and run the simulation until $T = 0.2$.
Figure~\ref{fig:sod} shows the results obtained using four
different types of discretizations, where ESDG refers to the nodal entropy
stable DG discretizations mentioned in~\cite{lin2023positivity}. We observe that the
DGSEM solution clearly violates the entropy condition, resulting in a jump
discontinuity at the rarefaction wave. On the other hand, the other three
discretizations, including the proposed limiting strategy using cell entropy
inequality, do not violate the entropy condition and produce a smooth
rarefaction wave.

\begin{figure}[!htb]
\centering
\begin{subfigure}{.5\textwidth}
  \centering
  \includegraphics[width=.9\linewidth]{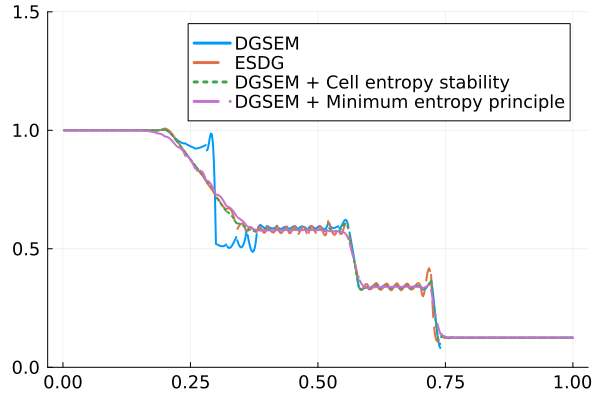}
  \caption{Solutions, $N = 3, K = 50$}
\end{subfigure}%
\begin{subfigure}{.5\textwidth}
  \centering
  \includegraphics[width=.9\linewidth]{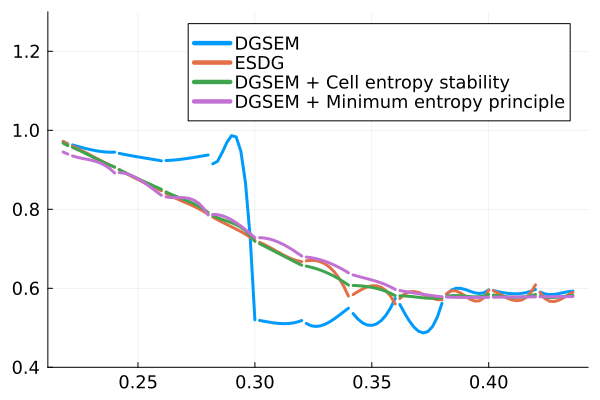}
  \caption{Zoom in view, $N = 3, K = 50$}
\end{subfigure}
\newline
\begin{subfigure}{.5\textwidth}
  \centering
  \includegraphics[width=.9\linewidth]{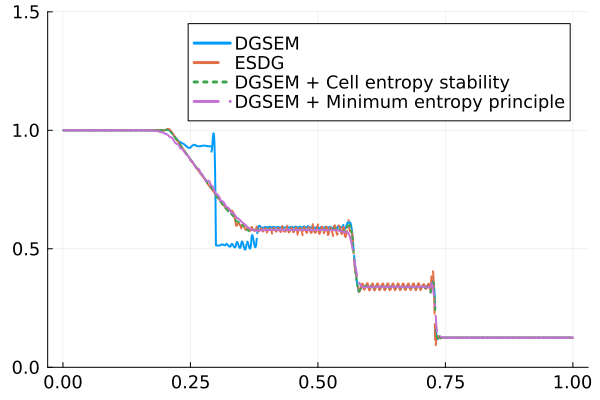}
  \caption{Solutions, $N = 3, K = 100$}
\end{subfigure}%
\begin{subfigure}{.5\textwidth}
  \centering
  \includegraphics[width=.9\linewidth]{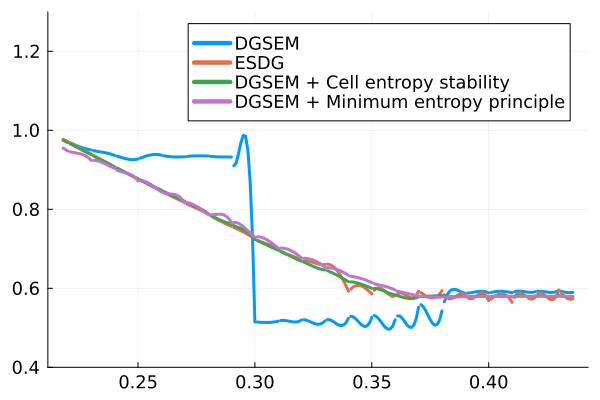}
  \caption{Zoom in view, $N = 3, K = 100$}
\end{subfigure}
\caption{Modified Sod shocktube }\label{fig:sod}
\end{figure}

\subsubsection{2D KPP}~\label{sec:2DKPP}
We next consider the 2D KPP problem~\cite{kurganov2007adaptive}, where typical
high-order methods like DGSEM will converge to non-entropic solutions. We
set up the problem using the same initial condition as
in~\cite{rueda2022subcell}. 
\rone{
We adopt the blending function and parameters in~\cite{hennemann2021provably} for shock
capturing to eliminate oscillations near shocks for the DGSEM discretizations.
In particular, we define the elementwise blending function $\alpha \in \LRs{0,1}$, with
\begin{align}\label{eq:blendingfunc}
  \alpha = \frac{1}{1 + \exp\LRp{\frac{-s}{\tau}\LRp{\epsilon - \tau}}}, \qquad s = \ln\LRp{\frac{1-0.0001}{0.0001}}, \qquad \tau = 0.5 \cdot 10^{-1.8 \LRp{N+1}^{0.25}}
\end{align}
where $\epsilon$ is the modal smoothness factor as defined in~\eqref{eq:elemmodalsmooth},
$s$ is the sharpness factor, and $\tau$ is a threshold value. Then, we apply shock
capturing by upper bounding the subcell limiting factors $\bar{l}_i$, who are
solutions of~\eqref{eq:LP2D}, over each element
by the elementwise blending function~\cite{lin2023positivity}:
\begin{align}\label{eq:shockcapblending}
  l_i = \min \LRp{\bar{l}_i, 1-\alpha},
\end{align}
so that an arbitrary amount of the low-order method may be blended in near
shocks, or when $\alpha$ is near $1$.
}

The domain $\LRs{0,2}\times\LRs{0,2}$ is discretized into uniform quadrilateral
elements by dividing the $x$ and $y$ directions with $K_{\oned}, K_{\oned}$
uniform intervals. We plot the values in the range $\LRs{-0.5,12}$.
Figure~\ref{fig:kpp} shows the solutions
obtained using DGSEM with and without the proposed entropy limiting when $N
= 3$. \rone{We refine the mesh from $K_{\oned} = 64$ to $K_{\oned} = 128$ to
check the convergence behaviour. We observe that DGSEM without the proposed
entropy limiter converges to a non-entropic solution with nonentropic artifacts,
which is similarly observed in~\cite{rueda2022subcell}. On the other
hand, the addition of the proposed cell entropy limiter results in a correct
entropic solution.}

\begin{figure}[!htb]
\centering
\begin{subfigure}{\textwidth}
  \centering
  \includegraphics[width=.7\linewidth]{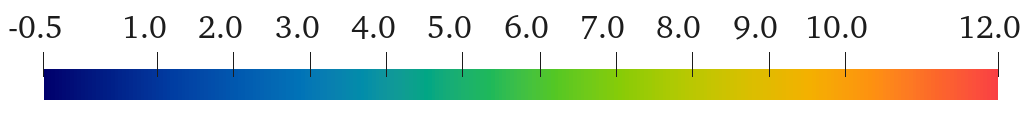}
\end{subfigure}%
\par\bigskip
\centering
\begin{subfigure}{.5\textwidth}
  \centering
  \includegraphics[width=.7\linewidth]{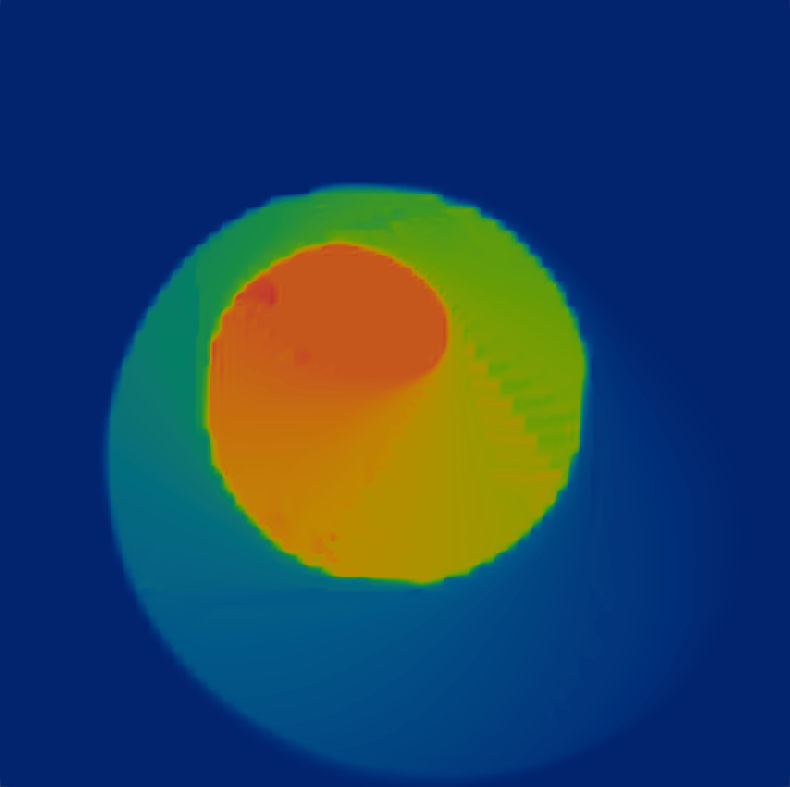}
  \caption{DGSEM,  $N = 3, K = 64$}
\end{subfigure}%
\begin{subfigure}{.5\textwidth}
  \centering
  \includegraphics[width=.7\linewidth]{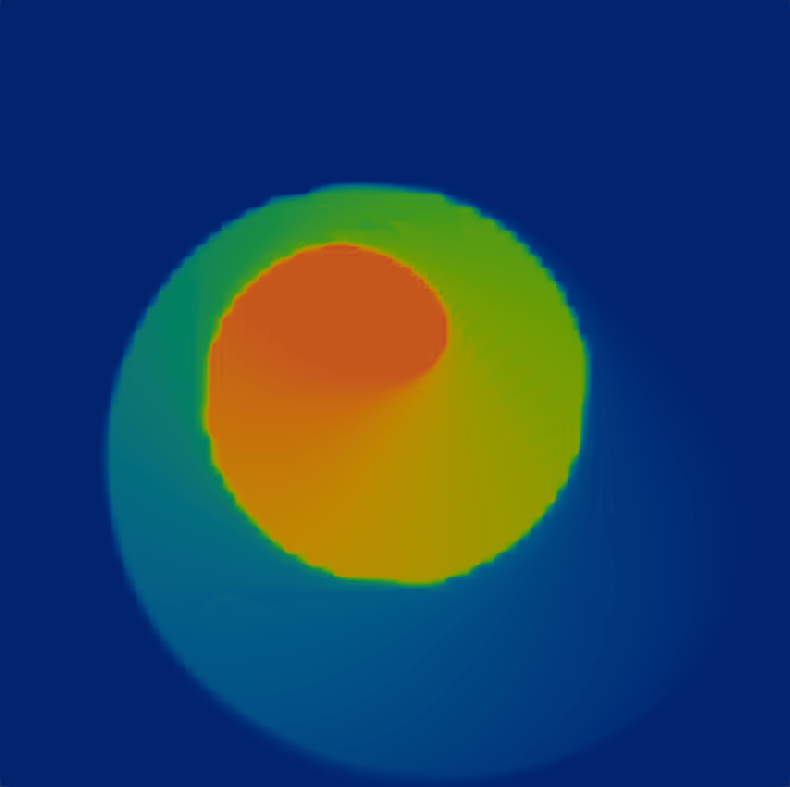}
  \caption{DGSEM with cell entropy inequality, $N = 3, K = 64$}
\end{subfigure}
\newline
\vspace{1cm}
\begin{subfigure}{.5\textwidth}
  \centering
  \includegraphics[width=.7\linewidth]{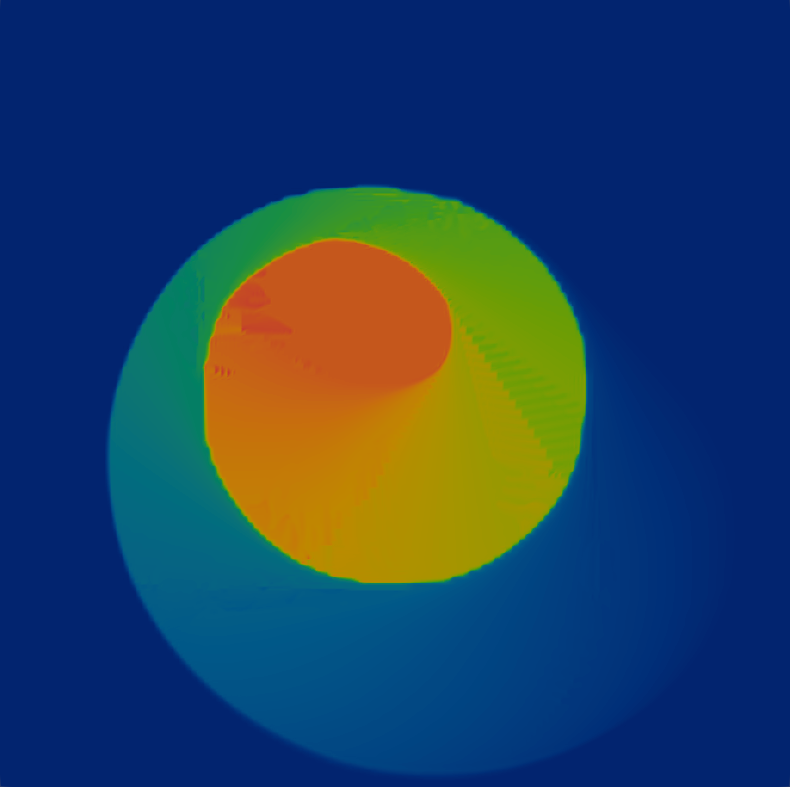}
  \caption{DGSEM, $N = 3, K = 128$}
\end{subfigure}%
\begin{subfigure}{.5\textwidth}
  \centering
  \includegraphics[width=.7\linewidth]{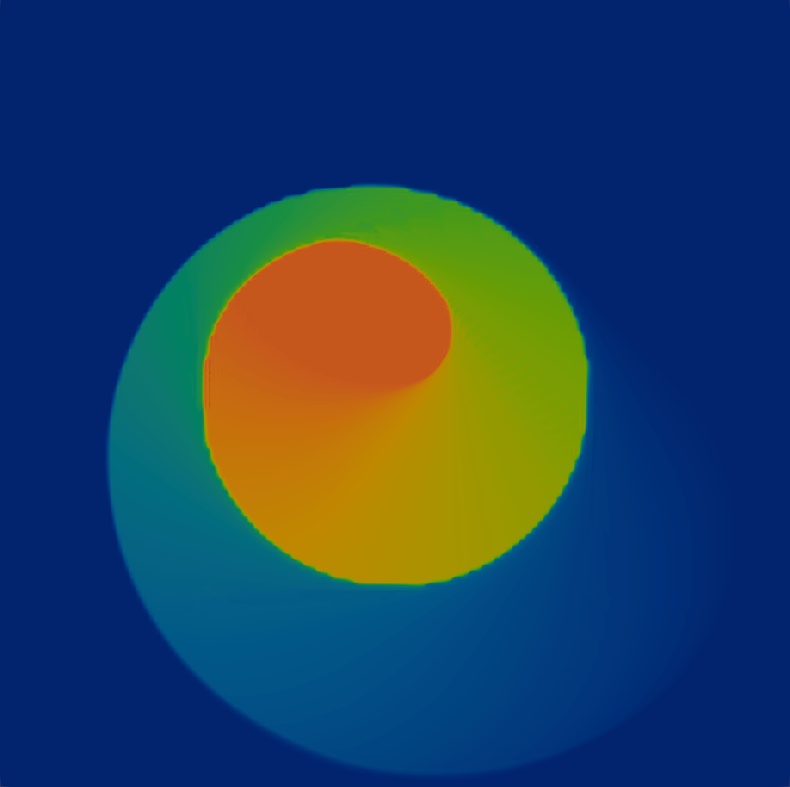}
  \caption{DGSEM with cell entropy inequality, $N = 3, K = 128$}
\end{subfigure}
\caption{2D KPP problem}\label{fig:kpp}
\end{figure}

\subsection{High order accuracy}~\label{sec:highorderexperiment}
\rone{In this section, we verify
the convergence of the subcell limited solution with cell entropy inequality and
no additional convex constraints. In particular, we assume $l^{\rm C} = 1$. }
We examine the convergence of the limited solution using the proposed strategy
in 2D with the isentropic vortex test case for the compressible Euler
equation~\cite{shu1988efficient}. Details of the analytical solution can be
found in Section 8.2.1 of~\cite{lin2023positivity}. The strength of the vortex
is set to $\beta_{\rm vortex} = 5.0$, such that no positivity limiting is needed.
Periodic boundary conditions are imposed, and the simulation is run until $T =
1.0$. The domain $\LRs{0,20}\times\LRs{0,10}$ is
decomposed into uniform quadrilateral elements by discretizing the $x$ and $y$
directions with $2K_{\oned}$ and $K_{\oned}$ uniform intervals, respectively.

We evaluate the relative $L^2$ errors in the conservative variables using
\rone{LGL} quadrature:
\begin{align}
    \frac{\LRs{\fnt{1}^T\fnt{M}\LRp{\vec{\rho}^n -\vec{\rho}}^2}^{1/2}}{\LRs{\fnt{1}^T\fnt{M}\vec{\rho}^2}^{1/2}} + \frac{\LRs{\fnt{1}^T\fnt{M}\LRp{\vec{\rho u}^n -\vec{\rho u}}^2}^{1/2}}{\LRs{\fnt{1}^T\fnt{M}\LRp{\vec{\rho u}}^2}^{1/2}}+\frac{\LRs{\fnt{1}^T\fnt{M}\LRp{\vec{\rho v}^n -\vec{\rho v}}^2}^{1/2}}{\LRs{\fnt{1}^T\fnt{M}\LRp{\vec{\rho v}}^2}^{1/2}} +     \frac{\LRs{\fnt{1}^T\fnt{M}\LRp{\vec{E}^n -\vec{E}}^2}^{1/2}}{\LRs{\fnt{1}^T\fnt{M}\vec{E}^2}^{1/2}} , \label{eq:error}
\end{align}
where the numerical solutions and exact solutions evaluated at quadrature nodes
are denoted by $\vec{u}^n$ and $\vec{u}$ respectively.

Table~\ref{tab:vortexcellES} shows that the subcell limited solution with cell
entropy stability yields an asymptotic convergence rate between $O\LRp{h^{N+1/2}}$
and $O\LRp{h^{N+1}}$. On the other hand, Tables~\ref{tab:vortexminentropy}
and~\ref{tab:vortexrelaxminentropy} show that the minimum entropy
principle-limited solutions have at most an $O\LRp{h}$ convergence rate and at
most an $O\LRp{h^2}$ rate after relaxation for any polynomial order $N$~\footnote{By
private communication, the subcell entropy fix convergence rate reduces to
$O\LRp{h^2}$ for a smooth sine wave test problem~\cite{rueda2023monolithic}}.
\rtwo{We note that the proposed strategy of enforcing cell entropy stability
does not provably preserve high order accuracy for smooth problems. This
section only provides numerical evidence of high-order accuracy.}

\begin{table}[!ht]
\centering
\scalebox{0.88}{
\begin{tabular}{|c|c|c|c|c|c|c|c|c|}
  \hline
 & \multicolumn{2}{c|}{$N=1$} & \multicolumn{2}{c|}{$N=2$} & \multicolumn{2}{c|}{$N=3$} & \multicolumn{2}{c|}{$N=4$}\\
  \hline
 K & $L^2$ error & Rate & $L^2$ error & Rate & $L^2$ error & Rate & $L^2$ error & Rate \\  
  \hline
  5  & $1.183\times 10^{ 0}$ &      & $6.935\times 10^{-1}$ &      & $2.498\times 10^{-1}$ &      & $1.587\times 10^{-1}$ &     \\ 
 10  & $7.722\times 10^{-1}$ & 0.62 & $1.785\times 10^{-1}$ & 1.96 & $7.083\times 10^{-2}$ & 1.82 & $2.000\times 10^{-2}$ & 2.99\\
 20  & $3.327\times 10^{-1}$ & 1.22 & $4.126\times 10^{-2}$ & 1.11 & $8.898\times 10^{-3}$ & 2.99 & $9.557\times 10^{-4}$ & 4.39\\
 40  & $1.118\times 10^{-1}$ & 1.57 & $6.714\times 10^{-3}$ & 2.62 & $8.163\times 10^{-4}$ & 3.45 & $3.142\times 10^{-5}$ & 4.93\\
 80  & $3.133\times 10^{-2}$ & 1.84 & $1.210\times 10^{-3}$ & 2.74 & $4.208\times 10^{-5}$ & 4.28 & $1.530\times 10^{-6}$ & 4.36\\
  \hline
\end{tabular}
}
\caption{2D isentropic vortex, uniform quadrilateral mesh, subcell limiting by enforcing cell entropy inequality}\label{tab:vortexcellES}
\end{table}

\begin{table}[!ht]
\centering
\scalebox{0.88}{
\begin{tabular}{|c|c|c|c|c|c|c|c|c|}
  \hline
 & \multicolumn{2}{c|}{$N=1$} & \multicolumn{2}{c|}{$N=2$} & \multicolumn{2}{c|}{$N=3$} & \multicolumn{2}{c|}{$N=4$}\\
  \hline
 K & $L^2$ error & Rate & $L^2$ error & Rate & $L^2$ error & Rate & $L^2$ error & Rate \\  
  \hline
  5  & $1.084\times 10^{ 0}$ &      & $7.498\times 10^{-1}$ &      & $4.499\times 10^{-1}$ &      & $3.135\times 10^{-1}$ &     \\ 
 10  & $7.012\times 10^{-1}$ & 0.63 & $3.343\times 10^{-1}$ & 1.17 & $2.109\times 10^{-1}$ & 1.09 & $1.486\times 10^{-1}$ & 1.08\\
 20  & $3.373\times 10^{-1}$ & 1.06 & $1.894\times 10^{-1}$ & 0.82 & $1.092\times 10^{-1}$ & 0.95 & $7.509\times 10^{-2}$ & 0.98\\
 40  & $1.841\times 10^{-1}$ & 0.87 & $9.718\times 10^{-2}$ & 0.96 & $5.956\times 10^{-2}$ & 0.87 & $4.160\times 10^{-2}$ & 0.85\\
 80  & $1.015\times 10^{-1}$ & 0.86 & $5.116\times 10^{-2}$ & 0.93 & $3.186\times 10^{-2}$ & 0.90 & $2.157\times 10^{-2}$ & 0.95\\
  \hline
\end{tabular}
}
\caption{2D isentropic vortex, uniform quadrilateral mesh, subcell limiting by enforcing minimum entropy principle}\label{tab:vortexminentropy}
\end{table}

\begin{table}[!ht]
\centering
\scalebox{0.88}{
\begin{tabular}{|c|c|c|c|c|c|c|c|c|}
  \hline
 & \multicolumn{2}{c|}{$N=1$} & \multicolumn{2}{c|}{$N=2$} & \multicolumn{2}{c|}{$N=3$} & \multicolumn{2}{c|}{$N=4$}\\
  \hline
 K & $L^2$ error & Rate & $L^2$ error & Rate & $L^2$ error & Rate & $L^2$ error & Rate \\  
  \hline
  5  & $1.086\times 10^{ 0}$ &      & $7.433\times 10^{-1}$ &      & $4.098\times 10^{-1}$ &      & $2.856\times 10^{-1}$ &     \\ 
 10  & $7.047\times 10^{-1}$ & 0.62 & $2.908\times 10^{-1}$ & 1.35 & $1.785\times 10^{-1}$ & 1.20 & $1.109\times 10^{-1}$ & 1.36\\
 20  & $3.154\times 10^{-1}$ & 1.16 & $1.203\times 10^{-1}$ & 1.27 & $5.464\times 10^{-2}$ & 1.71 & $3.352\times 10^{-2}$ & 1.73\\
 40  & $1.295\times 10^{-1}$ & 1.28 & $4.322\times 10^{-2}$ & 1.48 & $1.863\times 10^{-2}$ & 1.55 & $1.532\times 10^{-2}$ & 1.13\\
 80  & $4.500\times 10^{-2}$ & 1.52 & $1.869\times 10^{-2}$ & 1.21 & $6.614\times 10^{-3}$ & 1.49 & $6.120\times 10^{-3}$ & 1.31\\
  \hline
\end{tabular}
}
\caption{2D isentropic vortex, uniform quadrilateral mesh, subcell limiting by enforcing relaxed minimum entropy principle}\label{tab:vortexrelaxminentropy}
\end{table}


\rone{
\subsection{Kelvin-Helmholtz Instability}\label{sec:kelvinhelmholtz}

We now consider the Kelvin-Helmholtz instability~\cite{chan2022entropy}
to study the behaviour of the proposed limiting strategy in presence of turbulence.
The domain is $\LRs{-1, 1}^2$ with initial condition:
\begin{align}\label{eq:khinitial}
  \vec{u}\LRp{x} = \begin{bmatrix}\rho \\ u \\ v \\ p \end{bmatrix} = \begin{bmatrix} 0.5 + 0.75B\LRp{x,y} \\ 0.5\LRp{B\LRp{x,y}-1} \\ 0.1\sin\LRp{2\pi x} \\ 1\end{bmatrix}, \qquad B\LRp{x,y} = \tanh\LRp{15 y + 7.5} - \tanh\LRp{15 y - 7.5}.
\end{align}

We approximate the solution using degree $N = 3$ polynomials, and the domain is
discretized into $80 \times 80$ uniform quadrilateral elements. We ran the
simulation until $T = 10.0$, and plotted the density in the range of $\LRs{0.5, 2.5}$
in a logarithmic scale for better visibility. For this set of simulations, we 
enforce the cell entropy inequality~\eqref{eq:esstatement2D} and relaxed positivity
conditions on the density and internal energy
\begin{align}\label{eq:relaxedpos}
  \rho\LRp{\fnt{u}_i} \geq 0.5 \rho\LRp{\fnt{u}^{\low}_i}, \qquad \rho e\LRp{\fnt{u}_i} \geq 0.5 \rho e\LRp{\fnt{u}^{\low}_i},
\end{align}
using subcell limiting. 

We follow~\cite{pazner2021sparse,rueda2023monolithic,rueda2022subcell} to enforce
the bounds~\eqref{eq:relaxedpos}. We illustrate the
limiting procedure in 1D for brevity. The subcell limited solution~\eqref{eq:subcelllimsol}
can be rewritten as a convex combination of substates:
\begin{align}\label{eq:subcellsubstate}
  \fnt{u}_{i}^{n+1} = \frac{1}{2}\underbrace{\LRs{\fnt{u}_i^{\low} + l_i \frac{\Delta t}{\fnt{m}_i}\LRp{\bar{\fnt{f}}_i^{\high} - \bar{\fnt{f}}_i^{\low}}}}_{\mathcal{A}_i\LRp{l_i}} + \frac{1}{2}\underbrace{\LRs{\fnt{u}_i^{\low} + l_{i-1} \frac{\Delta t}{\fnt{m}_i}\LRp{\bar{\fnt{f}}_{i-1}^{\high} - \bar{\fnt{f}}_{i-1}^{\low}}}}_{\mathcal{B}_i\LRp{l_{i-1}}},
\end{align}
If the substates satisfy convex constraints, the limited solution will also satisfy
the constraints. Then for each node $i$, we need to find the maximum subcell limiting factor
$l_i^{\rm C}\in \LRs{0,1}$ that satisfies the constraints:
\begin{align}
  \rho\LRp{\mathcal{A}_i\LRp{l_i^{\rm C}}} &\geq 0.5 \rho \LRp{\fnt{u}_i^{\low}}, \qquad \ \ \rho\LRp{\mathcal{B}_{i+1}\LRp{l_i^{\rm C}}} \geq 0.5 \rho \LRp{\fnt{u}_{i+1}^{\low}},\label{eq:posrho}\\
  \rho e\LRp{\mathcal{A}_i\LRp{l_i^{\rm C}}} &\geq 0.5 \rho e \LRp{\fnt{u}_i^{\low}}, \qquad \rho e\LRp{\mathcal{B}_{i+1}\LRp{l_i^{\rm C}}} \geq 0.5 \rho e \LRp{\fnt{u}_{i+1}^{\low}}\label{eq:posrhoe}.
\end{align}
The constraints on density~\eqref{eq:posrho} are linear inequalities with a
closed-form solution, and the constraints on internal energy~\eqref{eq:posrhoe}
can be transformed into quadratic inequalities. Readers can refer
to~\cite{lin2023positivity} for the explicit formula of the limiting factors.

Figure~\ref{fig:kelvin-helmholtz-stddg} shows the proposed limiting strategy
applied to DGSEM. We observe the simulation is robust and resolving fine-scale
turbulence features. Additionally, we apply the proposed limiting strategy to
nodal entropy-stable discontinuous Galerkin (ESDG)\cite{lin2023positivity} for
comparison, as shown in Figure~\ref{fig:kelvin-helmholtz-esdg}. The proposed
limiting preserves the semi-discrete entropy inequality of ESDG discretizations.
Although the same types of constraints are applied through subcell limiting, the
turbulence structures of both solutions are visually different. In other words,
the choice of the high-order scheme in the proposed limiting strategy may have a
significant impact on the solution.

\begin{figure}[h]
\begin{subfigure}{\textwidth}
  \centering
  \includegraphics[width=.5\linewidth]{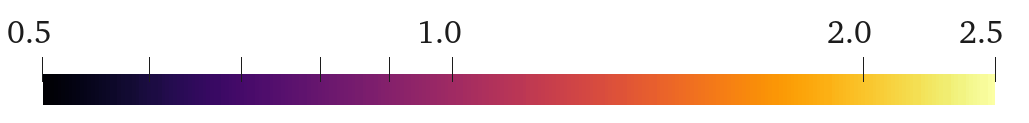}
\end{subfigure}
\par\bigskip
\centering
\begin{subfigure}{.3\textwidth}
  \centering
  \includegraphics[width=.9\linewidth]{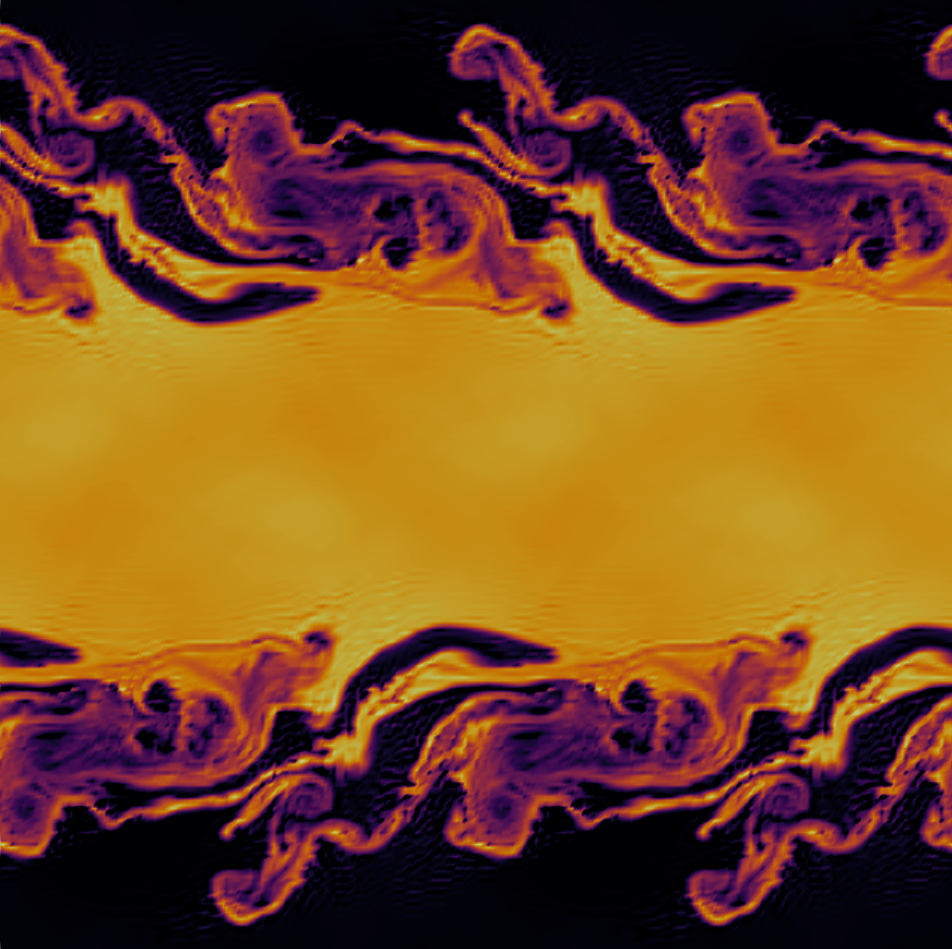}
  \caption{$T = 5.0$}
\end{subfigure}%
\begin{subfigure}{.3\textwidth}
  \centering
  \includegraphics[width=.9\linewidth]{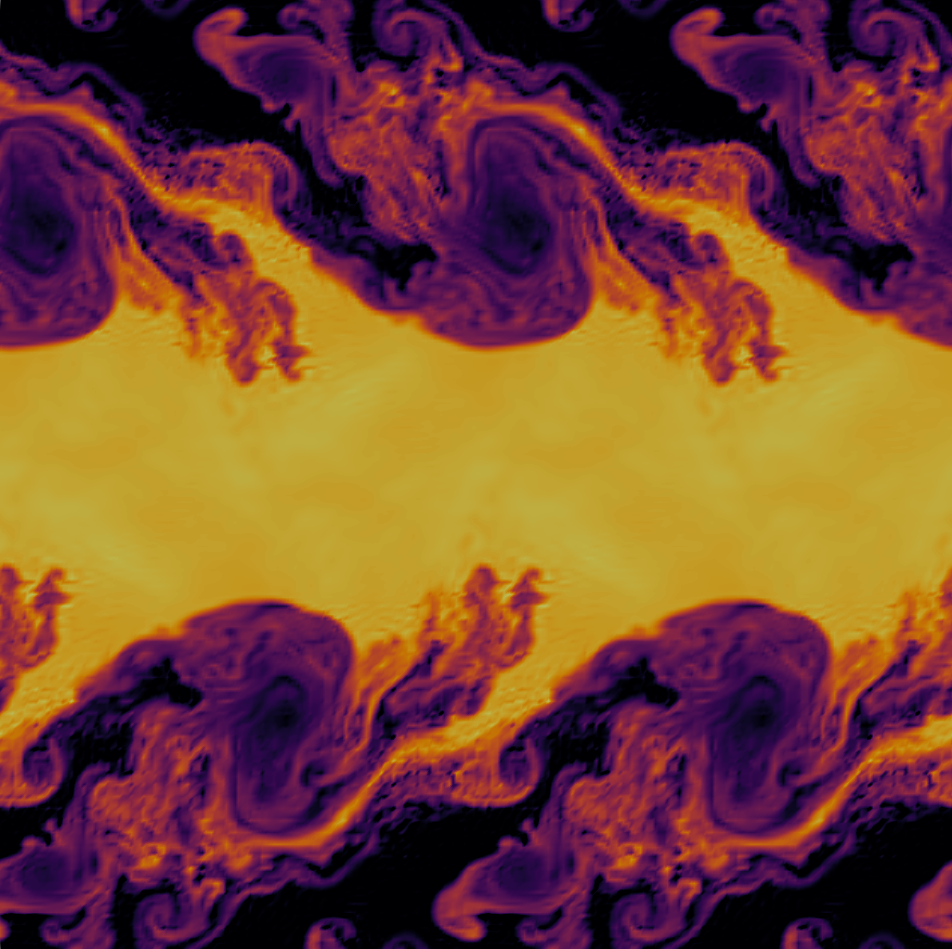}
  \caption{$T = 7.5$}
\end{subfigure}%
\begin{subfigure}{.3\textwidth}
  \centering
  \includegraphics[width=.9\linewidth]{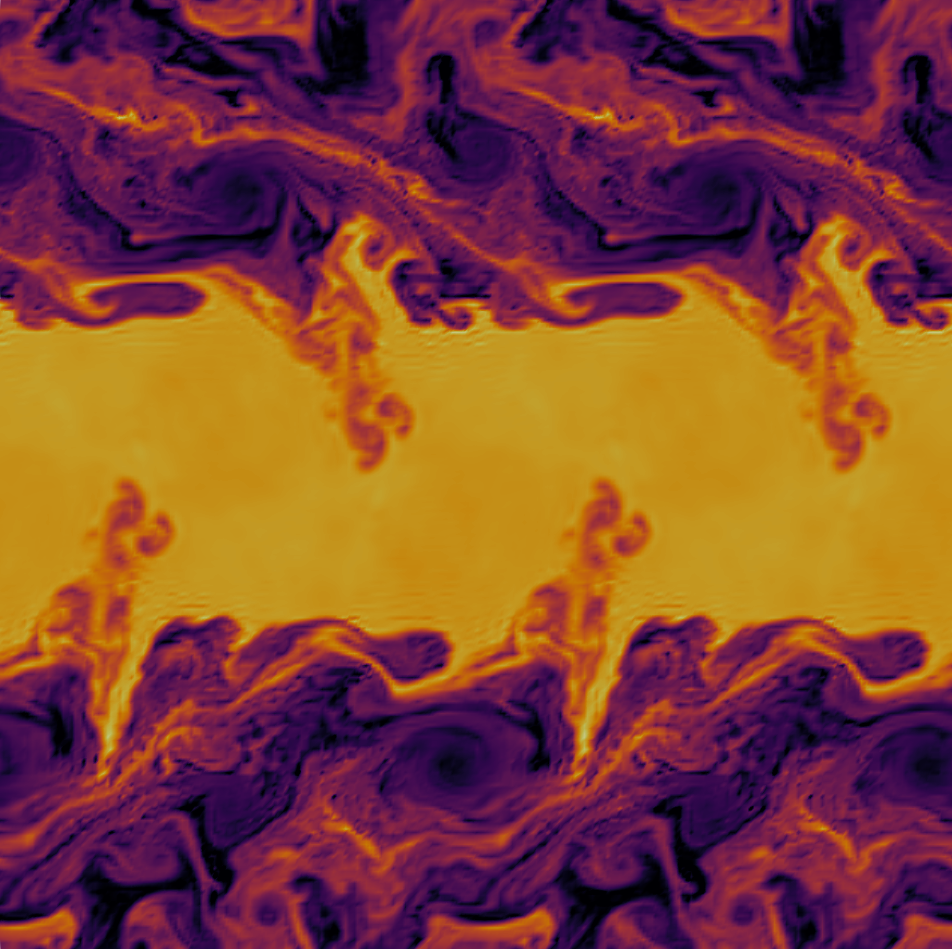}
  \caption{$T = 10.0$}
\end{subfigure}%
\caption{Kelvin-Helmholtz Instability, DGSEM limited with relaxed positivity constraints~\eqref{eq:relaxedpos} and the cell entropy inequality~\eqref{eq:esstatement2D} using subcell limiter}\label{fig:kelvin-helmholtz-stddg}
\end{figure}

\begin{figure}[h]
\begin{subfigure}{\textwidth}
  \centering
  \includegraphics[width=.5\linewidth]{fig/kh-bar.png}
\end{subfigure}
\par\bigskip
\centering
\begin{subfigure}{.3\textwidth}
  \centering
  \includegraphics[width=.9\linewidth]{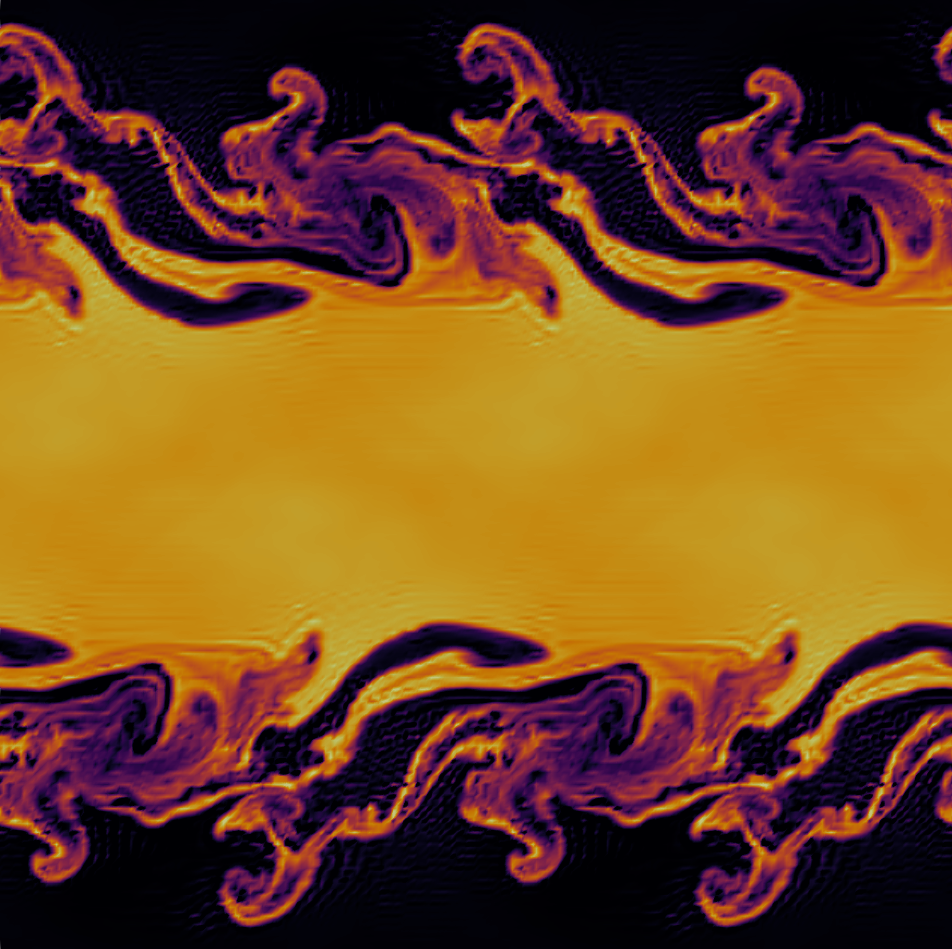}
  \caption{$T = 5.0$}
\end{subfigure}%
\begin{subfigure}{.3\textwidth}
  \centering
  \includegraphics[width=.9\linewidth]{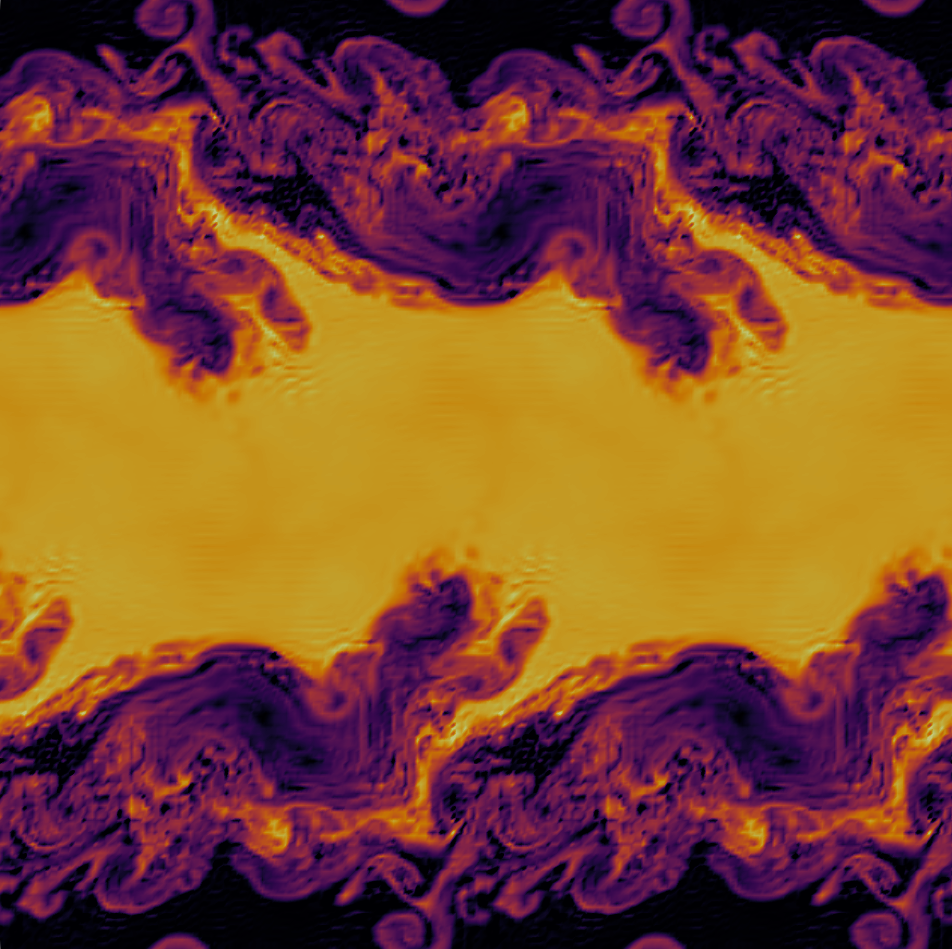}
  \caption{$T = 7.5$}
\end{subfigure}%
\begin{subfigure}{.3\textwidth}
  \centering
  \includegraphics[width=.9\linewidth]{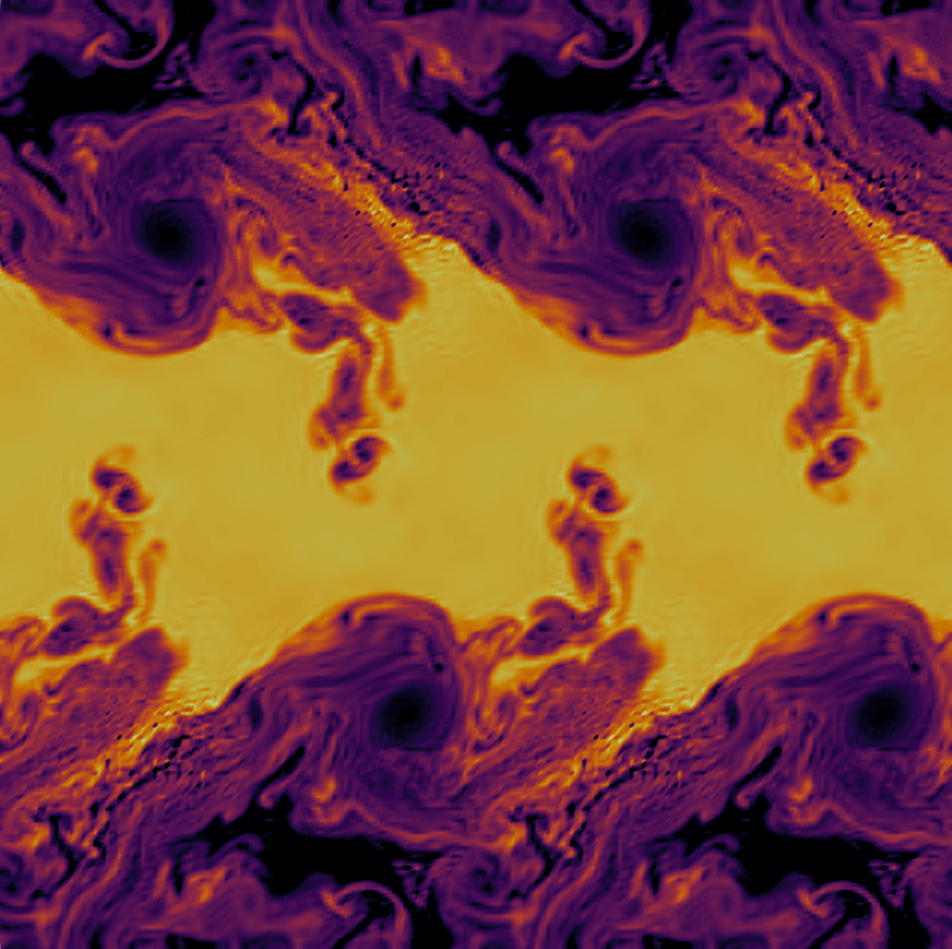}
  \caption{$T = 10.0$}
\end{subfigure}%
\caption{Kelvin-Helmholtz Instability, ESDG limited with relaxed positivity constraints~\eqref{eq:relaxedpos} and the cell entropy inequality~\eqref{eq:esstatement2D} using subcell limiter}\label{fig:kelvin-helmholtz-esdg}
\end{figure}

}

\subsection{Astrophysical jet}~\label{sec:astrojet}

We conclude the experiments by running the astrophysical jet test case proposed
by Ha et al.\cite{ha2005numerical} for the compressible Euler equation. This
test case involves a high Mach number of ${\rm Ma} \approx 2000$ and strong
shocks, making it ideal for testing the robustness of numerical
schemes\cite{rueda2023monolithic,zhang2010positivity}. The domain is
$\LRs{-0.5,0.5}^2$ initialized with a resting gas, and an
inflow boundary condition is imposed on the left of the domain.
\begin{align}
  \textrm{Initial: } \vec{u} = \begin{bmatrix}\rho\\ u \\ v \\ p\end{bmatrix} = \begin{bmatrix}0.5 \\ 0.0 \\ 0.0 \\ 0.4127\end{bmatrix}, \qquad \textrm{Inflow: } \vec{u}\LRp{-0.5,y} = \begin{bmatrix}\rho\\ u \\ v \\ p\end{bmatrix} = \begin{bmatrix}5.0 \\ 800.0 \\ 0.0 \\ 0.4127\end{bmatrix} \textrm{ for } y \in \LRs{-0.05,0.05}. \label{eq:astro-jet-ic}
\end{align}

We approximate the solution using degree $N = 3$ polynomials, and the domain was
discretized into $150\times 150$ uniform quadrilateral elements, as constructed
in Section~\ref{sec:2DKPP}. We plotted the density in the range of
$\LRs{0.01,30.0}$ in a logarithmic scale for better visibility. All simulations
were run until the end time of $T = 0.001$.

We first consider enforcing a TVD-like bound on density that
depends on the low order update
\begin{align}\label{eq:TVDbound}
  \rho\LRp{\fnt{u}_i} \in \LRs{\rho_i^{\min}, \rho_i^{\max}}, \qquad \rho_i^{\min} = \min\limits_{j \in \mathcal{N}\LRp{i}} \rho\LRp{\fnt{u}_i^{\low}} , \qquad \rho_i^{\max} = \max\limits_{j \in \mathcal{N}\LRp{i}} \rho\LRp{\fnt{u}_i^{\low}},
\end{align}
along with a relaxed positivity condition on the internal energy~\eqref{eq:relaxedpos} 
and the proposed cell entropy inequality~\eqref{eq:LP2Dwshockconstraintes}.
\rone{We follow Section~\ref{sec:kelvinhelmholtz} to enforce
the bounds~\eqref{eq:TVDbound} and~\eqref{eq:relaxedpos} with subcell limiting. Instead of the
relaxed positivity bound~\eqref{eq:posrho}, we enforce a TVD-like bound on density:
\begin{align}
  \rho\LRp{\mathcal{A}_i\LRp{l_i^{\rm C}}} &\in \LRs{\rho_i^{\min}, \rho_i^{\max}}, \qquad \rho\LRp{\mathcal{B}_{i+1}\LRp{l_i^{\rm C}}} \in \LRs{\rho_{i+1}^{\min}, \rho_{i+1}^{\max}},\label{eq:TVDrho}
\end{align}
which is a linear constraint with a closed form solution of $l_i^C$. The
constraints on density~\eqref{eq:TVDrho} and internal energy~\eqref{eq:posrhoe}
can be satisfied under explicit formulas in~\cite{lin2023positivity}.} Figure~\ref{fig:astrojet-tvd} shows
that the simulation remain robust in the presence of strong shocks.

Next, we study the effect of the relaxation factor $\beta$ in the cell entropy
inequality~\eqref{eq:esindicatorsubcell}. For this set of simulations, we only
enfroce relaxed positivity conditions on the density and internal energy~\eqref{eq:relaxedpos}
in addition to the shock capturing cell entropy inequality with different relaxation
factors $\beta$. Figure~\ref{fig:astrojet-dgsem-pos} compares the results with
$\beta = 1.0, 0.1, 0.01,$ and $0.0$. In other words, the results have at most
$100\%, 10\%, 1\%,$ and $0\%$ low-order cell entropy dissipation blended into the
enforced cell entropy bound.

As the relaxation factor $\beta$ decreases, finer-scale features are resolved
due to reduced numerical dissipation. However, we observe the appearance of the
carbuncle effect when $\beta = 0.01$ and $\beta = 0.0$, and similarly when using
the minimum entropy principle instead of the shock capturing cell entropy inequality, as
shown in Figure~\ref{fig:astrojet-min}. This suggests that the carbuncle effect
is not due only to a loss of entropy stability, but rather a lack of sufficient dissipation near
shocks. The qualitative behavior of the flow changes significantly with
different values of $\beta$, indicating a high sensitivity of the limited
solution to the choice of $\beta$.

\rone{
The final comparison is between
\begin{enumerate}
  \item The DGSEM discretization limited with the relaxed positivity bound~\eqref{eq:relaxedpos}
  and cell entropy inequality~\eqref{eq:esstatement2D} using a subcell limiter, as shown in Figure~\ref{fig:astrojet-dgsem-pos} when $\beta = 0$.
  \item The ESDG discretization limited with the relaxed positivity
bound~\eqref{eq:relaxedpos} using an elementwise Zhang-Shu limiter, as shown in Figure~\ref{fig:astrojet-esdg-zhangshu}.
  \item The ESDG discretization limited with the relaxed positivity bound and cell entropy
inequality~\eqref{eq:esstatement2D} using a subcell limiter, as shown in Figure~\ref{fig:astrojet-esdg-subcell}.
\end{enumerate}
All solutions satisfy the same relaxed positivity constraints and the cell
entropy stability~\cite{lin2023positivity}. The main difference is the limited
high-order scheme, DGSEM or ESDG, and the types of limiting, Zhang-Shu or
subcell, used. It is observed that the Zhang-Shu limited ESDG solution does not
exhibit the carbuncle effect that is present in the subcell limited solution
enforcing the cell entropy inequality. This can be partially explained by the
fact that the Zhang-Shu limiter is more dissipative than the subcell limiter,
and this numerical dissipation is sufficient to eliminate the carbuncle near the
shock in this particular test case.

Comparing Figure~\ref{fig:astrojet-dgsem-pos} when $\beta = 0$ and
Figure~\ref{fig:astrojet-esdg-subcell}, we observe different turbulence
structures when the proposed subcell limiting strategy is applied to DGSEM and
ESDG, respectively. This further supports our observation in
Section~\ref{sec:kelvinhelmholtz} that the choice of a high-order scheme in the
proposed limiting strategy may have a significant impact on the solution. }

\begin{figure}[!h]
\begin{subfigure}{\textwidth}
  \centering
  \includegraphics[width=.44\linewidth]{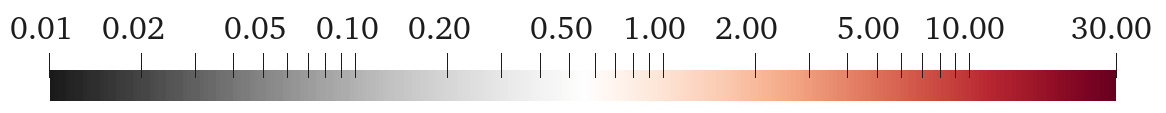}
\end{subfigure}
\par\medskip
\centering
\begin{subfigure}{\textwidth}
  \centering
  \includegraphics[width=.4\linewidth]{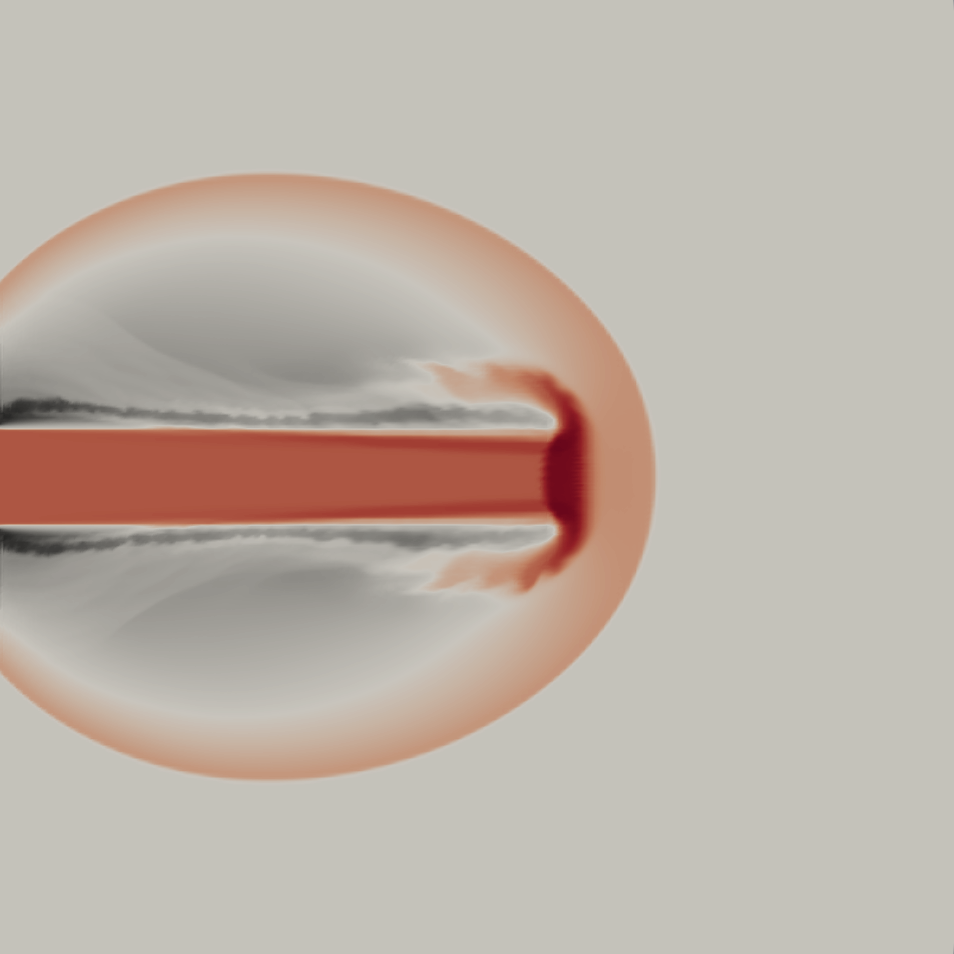}
\end{subfigure}
\caption{Astrophysical jet, DGSEM limited with TVD-like constraints~\eqref{eq:TVDrho},~\eqref{eq:posrhoe} and the cell entropy inequality~\eqref{eq:esstatement2D}}~\label{fig:astrojet-tvd}
\end{figure}

\begin{figure}[!h]
\begin{subfigure}{\textwidth}
  \centering
  \includegraphics[width=.7\linewidth]{fig/astro-bar-horizontal.png}
\end{subfigure}
\par\bigskip
\centering
\begin{subfigure}{.5\textwidth}
  \centering
  \includegraphics[width=.8\linewidth]{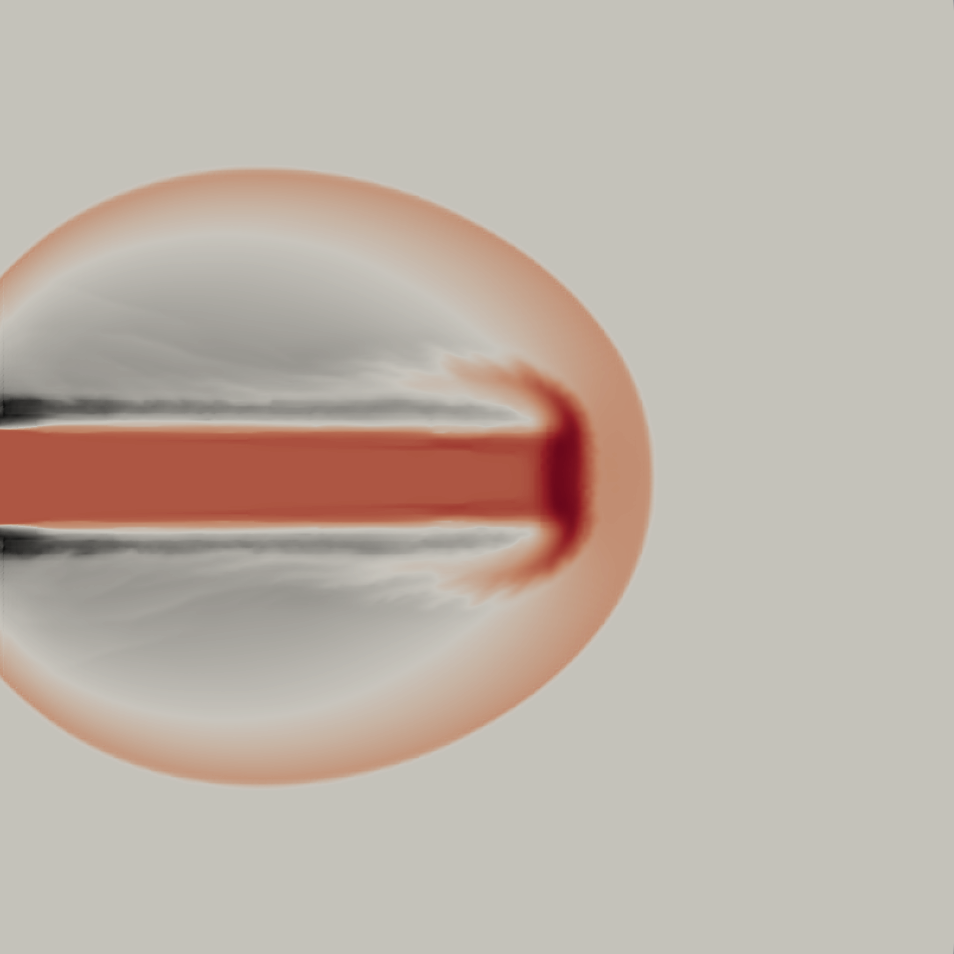}
  \caption{$\beta = 1.0$}
\end{subfigure}%
\begin{subfigure}{.5\textwidth}
  \centering
  \includegraphics[width=.8\linewidth]{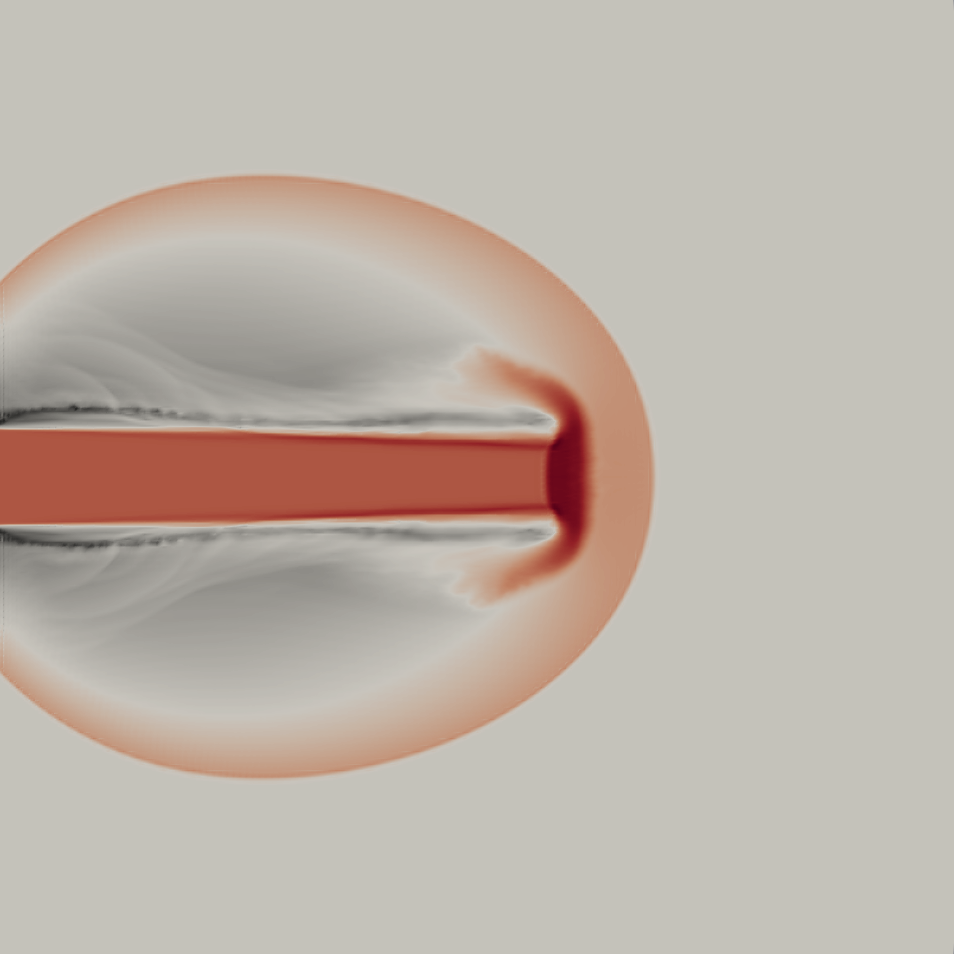}
  \caption{$\beta = 0.1$}
\end{subfigure}%
\newline
\vspace{1cm}
\begin{subfigure}{.5\textwidth}
  \centering
  \includegraphics[width=.8\linewidth]{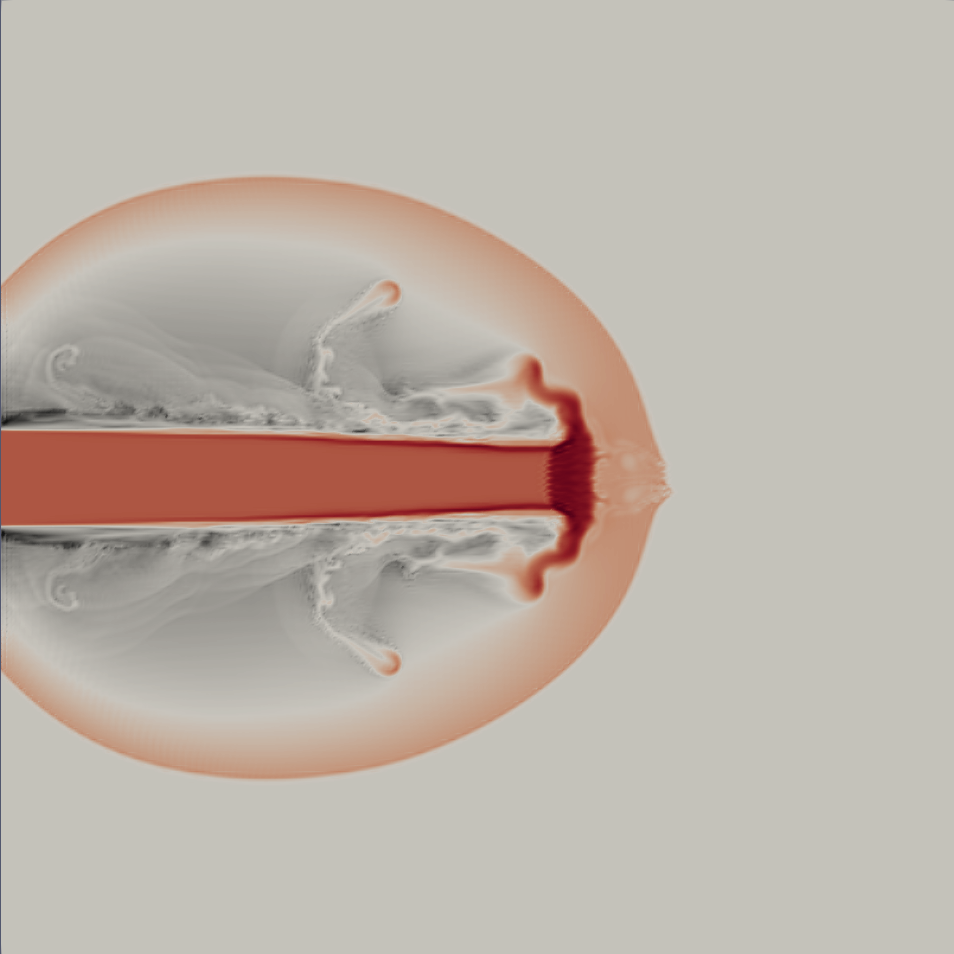}
  \caption{$\beta = 0.01$}
\end{subfigure}%
\begin{subfigure}{.5\textwidth}
  \centering
  \includegraphics[width=.8\linewidth]{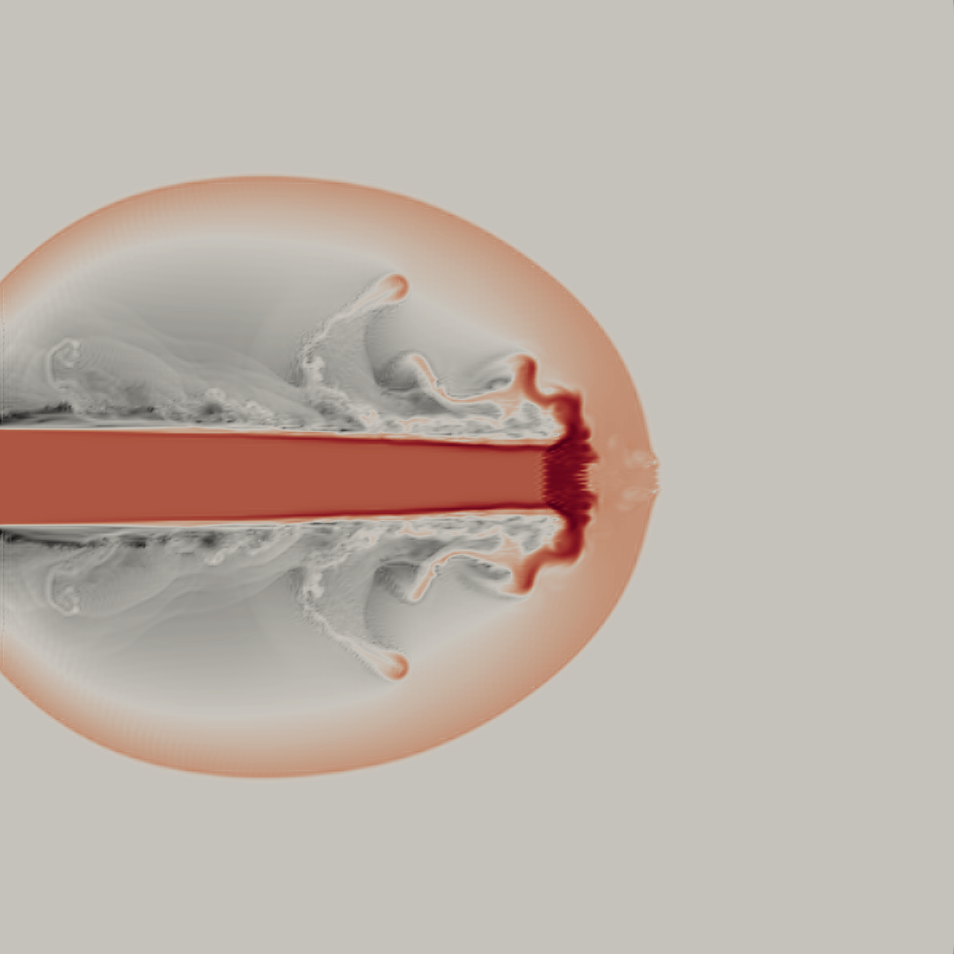}
  \caption{$\beta = 0.0$}
\end{subfigure}%
\caption{Astrophysical jet, DGSEM limited with relaxed positivity constraints~\eqref{eq:relaxedpos} and the shock capturing cell entropy inequality~\eqref{eq:esindicatorsubcell}}\label{fig:astrojet-dgsem-pos}
\end{figure}

\section{Conclusion}~\label{sec:conclusion} 
In this paper, we introduce a novel subcell limiting strategy to ensure
semi-discrete entropy stability. We formulate the limiting procedure as a linear
program that can be solved efficiently using a deterministic greedy algorithm.
The resulting subcell limited solution is high order accurate and
semi-discretely entropy stable, and can be combined with other subcell limiters
to preserve general convex constraints.

\section*{Acknowledgement}
\rzero{The authors thank Sebastian Perez-Salazar for pointing out the connection to the continuous
knapsack problem.} Yimin Lin and Jesse Chan gratefully acknowledge support from
National Science Foundation under award DMS-CAREER-1943186 \rzero{and DMS-2231482}.

\appendix

\section{Astrophysical jets}




\begin{figure}[h]
\begin{subfigure}{\textwidth}
  \centering
  \includegraphics[width=.5\linewidth]{fig/astro-bar-horizontal.png}
\end{subfigure}
\par\bigskip
\centering
\begin{subfigure}{.3\textwidth}
  \centering
  \includegraphics[width=.9\linewidth]{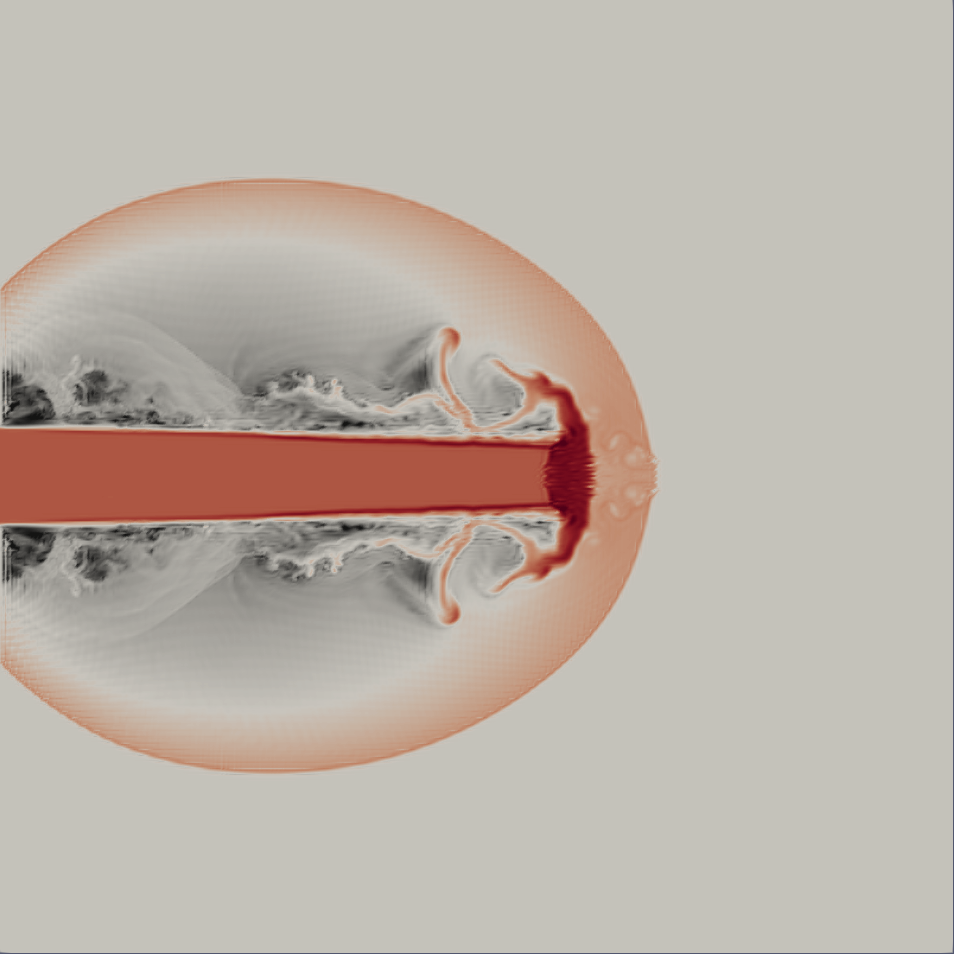}
  \caption{(a) ESDG limited with relaxed positivity constraint~\eqref{eq:relaxedpos} and cell entropy inequality~\eqref{eq:esstatement2D} using subcell limiter}\label{fig:astrojet-esdg-subcell}
\end{subfigure}\hfill%
\begin{subfigure}{.3\textwidth}
  \centering
  \includegraphics[width=.9\linewidth]{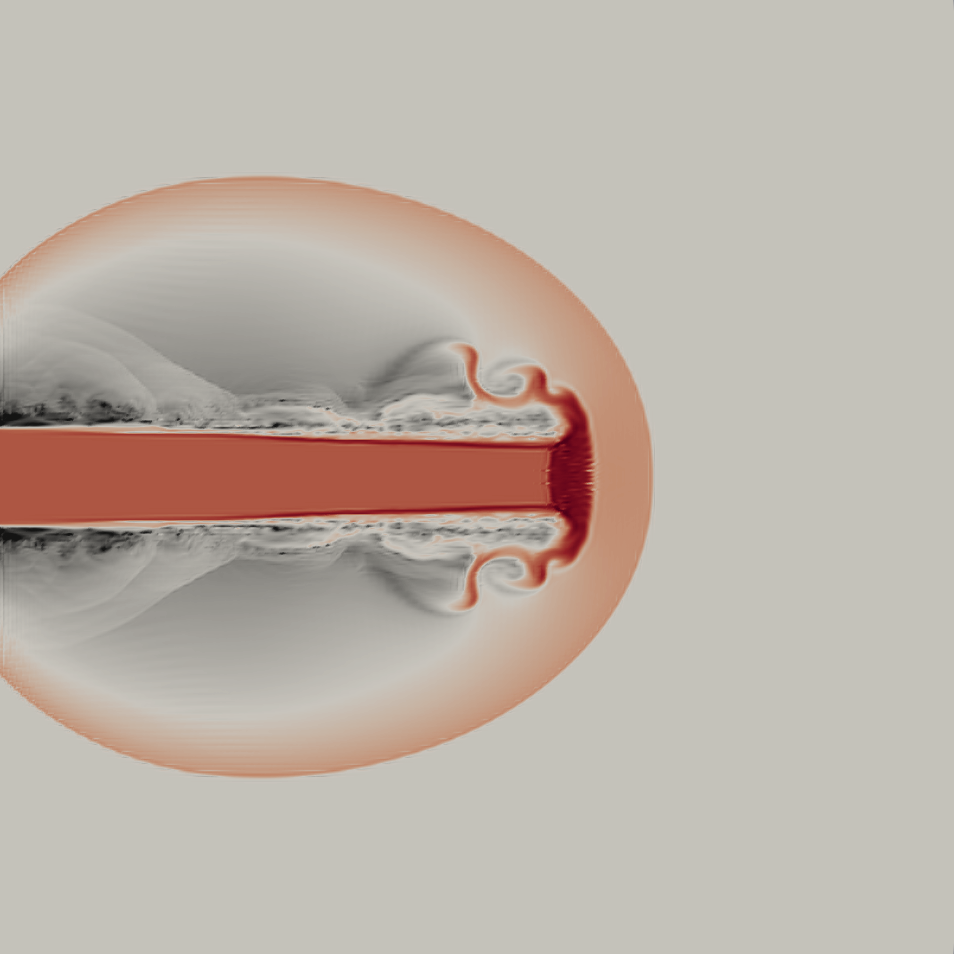}
  \caption{(b) ESDG limited with relaxed positivity constraint~\eqref{eq:relaxedpos} using Zhang-Shu limiter~\cite{lin2023positivity}}\label{fig:astrojet-esdg-zhangshu}
\end{subfigure}\hfill%
\begin{subfigure}{.3\textwidth}
  \centering
  \includegraphics[width=.9\linewidth]{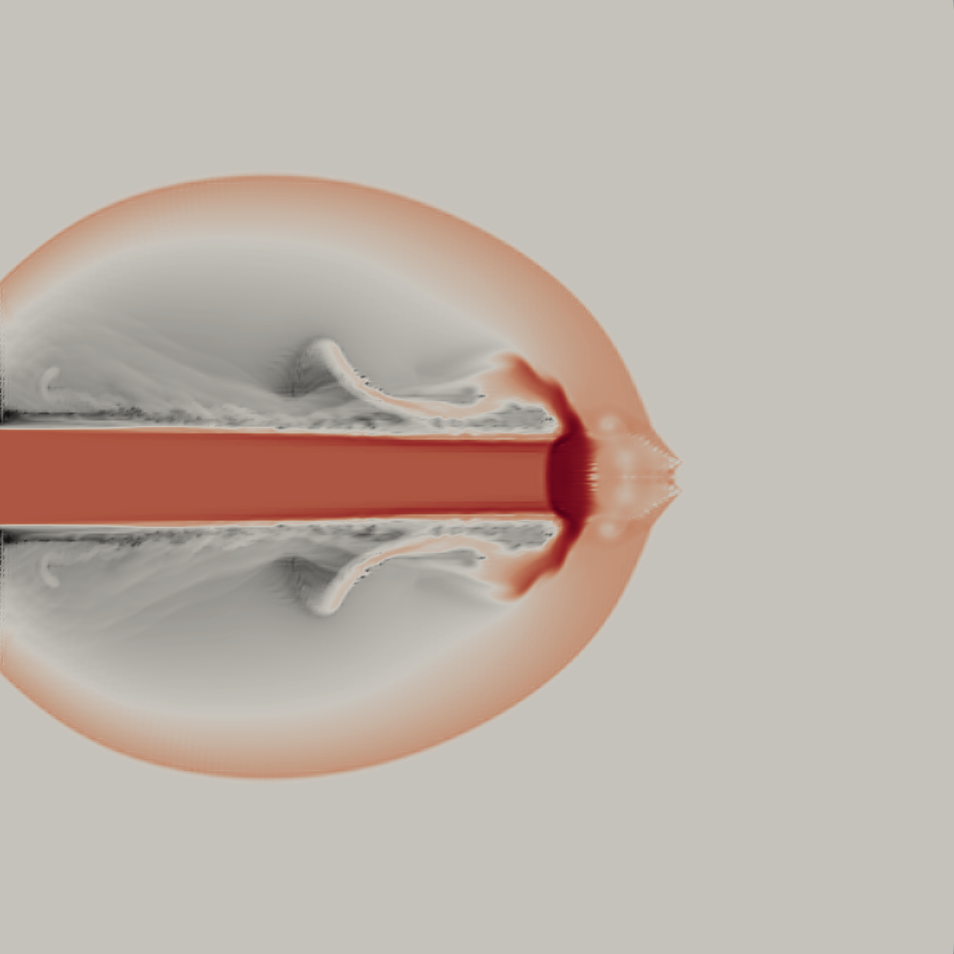}
  \caption{(c) DGSEM limited with relaxed positivity constraint~\eqref{eq:relaxedpos} and minimum entropy principle~\eqref{eq:minentropy} using subcell limiter}\label{fig:astrojet-min}
\end{subfigure}%
\caption{Astrophysical jet}\label{fig:astro-jet-appendix}
\end{figure}


\newpage
\bibliographystyle{elsarticle-num}
\bibliography{reference.bib}

\end{document}